\documentclass[hidelinks,onefignum,onetabnum]{siamart250211}

\usepackage{lipsum}
\usepackage{amsfonts}
\usepackage{graphicx}
\usepackage{epstopdf}
\usepackage{algorithmic}
\usepackage{bm}
\ifpdf
  \DeclareGraphicsExtensions{.eps,.pdf,.png,.jpg}
\else
  \DeclareGraphicsExtensions{.eps}
\fi

\newsiamremark{rem}{Remark}
\newsiamremark{ass}{Hypothesis}
\crefname{hypothesis}{Hypothesis}{Hypotheses}
\newsiamthm{thm}{Theorem}
\newsiamthm{claim}{Claim}
\newsiamremark{fact}{Fact}
\crefname{fact}{Fact}{Facts}

\headers{Event-Triggered Stabilisation of Desynchronisation}{L. Yang, J. Zhang, J. Lamb, V. Voon, E. Schöll and W. Lin}

\title{Event-Triggered Stabilisation of Desynchronisation in Networked Oscillatory Systems}

\author{Luan Yang\thanks{Research Institute of Intelligent Complex Systems, Fudan University, Shanghai, China, and Department of Psychiatry, University of Cambridge, Cambridge, UK
  (\email{luanyang23@m.fudan.edu.cn}).}
\and Jingdong Zhang\thanks{Department of Mathematics and I-X, Imperial College London, London, UK 
  (\email{j.zhang1@imperial.ac.uk}).}
  \and Jeroen Lamb\thanks{Department of Mathematics, Imperial College London, London, UK 
  (\email{jsw.lamb@imperial.ac.uk}).}
  \and Valerie Voon\thanks{Department of Psychiatry, University of Cambridge, Cambridge, UK 
  (\email{vv247@cam.ac.uk}).}
  \and Eckehard Schöll\thanks{Corresponding author. Technische Universität Berlin, Institut für Physik und Astronomie, 10623 Berlin, Germany, and Bernstein Center for Computational Neuroscience (BCCN) Berlin
  (\email{schoell@physik.tu-berlin.de}).}
  \and Wei Lin\thanks{School of Mathematical Sciences, Research Institute of Intelligent Complex Systems, Fudan University, Shanghai,
 China (\email{wlin@fudan.edu.cn}).} }

\usepackage{amsopn}

\ifpdf
\hypersetup{
  pdftitle={An Example Article},
  pdfauthor={D. Doe, P. T. Frank, and J. E. Smith}
}
\fi

\usepackage{subfigure}

\begin{document}

 \def\vzero{{\bm{0}}}
  \def\vz{{\bm{z}}}
  \def\vone{{\bm{1}}}
  \def\vmu{{\bm{\mu}}}
  \def\vtheta{{\bm{\theta}}}
  \def\vsigma{{\bm{\sigma}}}
  \def\vf{{\bm{f}}}
  \def\vg{{\bm{g}}}
  \def\vu{{\bm{u}}}
    \def\vh{{\bm{h}}}
  \def\vx{{\bm{x}}}
  \def\vZ{{\bm{Z}}}
  \def\vA{{\bm{A}}}
  \def\vy{{\bm{y}}}
  \def\vv{{\bm{v}}}
    \def\vp{{\bm{p}}}
  \def\vs{{\bm{s}}}
  \def\vq{{\bm{q}}}
  \def\vb{{\bm{b}}}
    \def\vX{{\bm{X}}}
  \def\vxi{{\bm{\xi}}}
  
\maketitle

\begin{abstract}
Pathological neuronal synchrony provides one practical motivation for studying sparse desynchronisation control in networked oscillatory systems, particularly in applications where actuation and communication are resource constrained, as in deep brain stimulation. Motivated by this challenge, we study how to stabilise desynchronisation in coupled oscillatory dynamical systems without continuously updated control, even when the oscillators' phase is unavailable.
We develop a general event-triggered control framework for stabilising the desynchronised state of coupled limit-cycle oscillatory dynamical systems. Our analysis establishes a unified theoretical result showing that desynchronisation can be achieved by various feedback controllers that act sparsely and depend only on an order parameter  observable. The proposed controllers admit a gradient-descent interpretation and stabilise the desynchronisation state of general phase-reduced oscillator networks. We further prove the controlled systems under event-triggered mechanisms possess a strictly positive lower bound of the inter-event dwell-time, excluding Zeno behaviour. 
To address the practical unavailability of exact phase reductions, we introduce a pseudo-phase construction that yields a computable order parameter from state measurements alone. Numerical studies on representative oscillator networks demonstrate the effectiveness and robustness of the proposed framework.
\end{abstract}



\section{Introduction}
 Synchronisation is a pervasive emergent phenomenon in networks of interacting dynamical units \cite{ROS01a,zhang2024machine,wu1995synchronization,wu2002synchronization,pan2019emergent,strogatz2005crowd}.  In human brain, however, excessive neuronal synchrony is closely associated with disorders such as Parkinson’s disease \cite{alberts1969cortical,lenz1994single,goldberg2004spike} and epilepsy \cite{GER20}. This has motivated growing interest in desynchronisation strategies, particularly in neuromodulation settings such as deep brain stimulation (DBS), where stimulation must operate under stringent sensing, actuation, and energy constraints \cite{2006dbs,little2014functional,tinkhauser2017beta,hammond2007pathological,wingeier2006intra,tinkhauser2017modulatory,eusebio2012does,KRO24}. From a control-theoretic perspective, these considerations lead naturally to a broader question: how to design resource-efficient feedback laws that stabilise desynchronised states in large-scale networks of coupled limit-cycle oscillators~\cite{yang2025advancements}?

Existing desynchronisation strategies for oscillatory systems include delayed feedback, adaptive control, and learning-based updates \cite{b4,b8,b14,E2016}. While these approaches have demonstrated effectiveness in specific settings, most of them are developed for particular models or control architectures and do not provide a unified framework for sparse desynchronisation control in general coupled oscillator networks. In particular, two main difficulties remain: (1) for coupled nonlinear and high-dimensional limit-cycle oscillators, it is nontrivial to construct feedback laws that both have low degree of freedom and guarantee convergence to desynchronised states beyond the classical Kuramoto setting; (2) under resource constraints, continuously updated feedback may be infeasible, so the control design must explicitly incorporate sparse communication mechanisms. Motivated by these issues, we study desynchronisation of general high-dimensional coupled oscillators under low-dimensional event-triggered feedback controllers acting directly in the original state space.

Phase provides a natural coordinate for analysing and controlling coupled oscillatory dynamics. Near a stable periodic orbit, phase reduction transforms high-dimensional oscillator models into tractable phase equations, thereby enabling feedback design in terms of collective synchrony \cite{monga2019phase}. This framework extends beyond homogeneous oscillators and applies to weakly perturbed and heterogeneous limit-cycle systems \cite{wilson2016isostable,wilson2018greater}, forming the basis of a broad class of phase-based control strategies \cite{efimov2009controlling,kiss2007engineering,wilson2014optimal}. In this work, we exploit the phase-reduced dynamics of coupled oscillators \cite{brown2004,ashwin2016mathematical,weerasinghe2019predicting} together with the associated order parameter \cite{acebron2005kuramoto} to construct static, dynamic, and mean-field feedback laws for desynchronisation. Since exact isochronous phases are typically unavailable in high-dimensional nonlinear systems, we further introduce a pseudo-phase construction that produces a computable order parameter directly from state measurements, thereby bridging the gap between phase-based analysis and feedback implementation on the original system.

In resource-constrained settings, continuously updated feedback may be infeasible because sensing, communication, and actuation must all be used sparingly. This motivates an event-triggered implementation, in which control updates are generated only at state-dependent instants rather than continuously or periodically \cite{heemels2012introduction,tabuada2007event}. For the proposed desynchronisation framework, the central analytical issue in the event-triggered setting is to prove that the inter-event times admit a strictly positive lower bound. Establishing such a dwell-time estimate is essential for excluding Zeno behaviour and ensuring the well-posedness of the triggering mechanism \cite{yang2025neural}.

This article is organized as follows: Section~\ref{sec preliminary} introduces the problem setting, notation, and phase-reduction framework for coupled limit-cycle oscillator networks. Section~\ref{sec desyn} develops order-parameter-based feedback laws for desynchronisation, including static, dynamic, and mean-field controllers. Section~\ref{sec etc} studies their event-triggered implementation and establishes positive lower bounds on the inter-event times, thereby excluding Zeno behaviour. Section~\ref{sec pseudo} introduces a pseudo-phase construction that yields computable order parameters from state measurements and connects the phase-based analysis back to the original oscillator dynamics. Section~\ref{sec exp} presents numerical experiments, and Section~\ref{sec conclusion} concludes the paper.

\section{Problem formulation and preliminaries}\label{sec preliminary}

We consider the following setting of a controlled network of coupled oscillators, 
\begin{equation}\label{eq main}
    \begin{aligned}
\dot{\vx}_j=\vf_j(\vx_j)+\sum_{k=1}^NA_{jk}\vh(\vx_k,\vx_j)+\vg(\vx_j)u_j,
    \end{aligned}
\end{equation}
where $\vx_j\in\mathbb{R}^n$, $j=1,...,N$ represents the state of an $n$-dimensional oscillator, $\vf_j$ represents the local dynamics of the oscillators, the weighted adjacency matrix $\vA=(A_{ij})\in\mathbb{R}^{n\times n}$  captures the interacting structure between oscillators, $\vh$ is the  $(i,j)$ pairwise interaction, $\vg\in\mathbb{R}^{n\times 1}$ is the actuator of controller $u_j\in\mathbb{R}$. The isolated oscillator
$\dot{\vx}_j = \vf_j(\vx),$  admits a stable $T_j$-periodic orbit with frequency $\omega_j = 2\pi/T_j$. 
Within the basin of attraction of the periodic orbit, there exists an isochronous phase function $\theta_j(\vx)$ that maps each state to a phase between $[0,2\pi]$, and an associated $T_j$-periodic phase response curve (PRC) defined by
$\vZ_j(\theta) = \partial \theta/\partial\vx_j$, satisfying $\vZ_j(\theta) \cdot \vf_j(\vx) = \omega_j$~\cite{monga2019phase}. The networked dynamics is assumed to admit a phase reduction in the sense of Hypothesis~\ref{ass reduction}.

\begin{ass}\label{ass reduction}
There exists $\varepsilon_0>0$ such that, for all $0<\varepsilon\le\varepsilon_0$, and for solutions remaining in the basin of attraction of the limit cycle $\vX$, the system admits a uniformly valid phase reduction of the form
\begin{equation}\label{eq phase reduction}
\dot{\theta}_j
=
\omega_j
+
\sum_{k=1}^N A_{jk}\,\Gamma(\theta_k,\theta_j)
+
b_j(\theta_j)\,u_j,
\end{equation}
Here the coupling function $\Gamma$ and the control gain $b_j$ are given by
\begin{equation*}
    \begin{aligned}
\Gamma(\theta_k,\theta_j)
&=
\vZ(\theta_j)\cdot
\vh\!\big(\vX(\theta_k),\vX(\theta_j)\big),\\
b_j(\theta)
&=
\vZ_j(\theta_j)\cdot \vg\big(\vX(\theta_j)\big),        
    \end{aligned}
\end{equation*}
where $\vX(\theta)$ denotes the inverse of the phase function $\theta(\vx)$ restricted on the limit cycle. The coupling function $\Gamma$ is assumed to be bounded and locally Lipschitz. Moreover, the periodic actuator gain is assumed to be nondegenerate, i.e.,
\[
|b_j(\theta)|>0
\quad
\text{for all $\theta$ and all $j$.}
\]
\end{ass}

\begin{rem}
Although the networked dynamics can be reduced to phase equations under appropriate assumptions, the accurate expression of the phase function and the corresponding PRC is generally intractable for nonlinear and high-dimensional systems. To address this issue, in Section~\ref{sec pseudo} we introduce a projection-based approach to approximate the phase variable.
\end{rem}

The degree of collective synchrony in the network is quantified by the Kuramoto order parameter.

\begin{definition}
Let $\theta_j(t)\in[0,2\pi)$ denote the phase of node $j$. The Kuramoto order parameter is defined as
\begin{equation}\label{eq:order-phase}
R(t)=\frac{1}{N}\sum_{j=1}^N e^{\mathrm{i}\theta_j(t)},
\end{equation}
We further define the synchronisation strength of the coupled system as
\[
Q(t)=|R(t)|^2.
\]
\end{definition}

The order parameter $R(t)$ measures the degree of phase coherence across the network. In particular, synchronisation corresponds to $Q(t)\to1$, whereas desynchronisation is characterised by $Q(t)\ll1$ or $Q(t)\to0$.

In this work, we focus on the \emph{phase desynchronisation problem} for network-coupled oscillators. Specifically, given the coupled dynamics \eqref{eq main} with its associated order parameter \eqref{eq:order-phase}, the objective is to design a feedback control law $u(t)$ such that
\begin{equation}\label{eq:desync-objective}
    \lim_{t\to\infty} |r(t)| = 0,
\end{equation}
thereby eliminating the collective synchrony in the network.

The control design should meet two essential requirements:  
\begin{itemize}
    \item {Efficiency:} achieve desynchronisation with event-triggered mechanism to reduce the communication cost of updating control signals.  
    \item {Generalization:} remain effective across different oscillator models and network topologies.
\end{itemize}

\section{Control laws for desynchronisation}\label{sec desyn}
{\color{black}
To devise feedback laws for the desynchronisation task and to analyse the asymptotic behaviour of the controlled phase system~\eqref{eq phase reduction}, we next consider the general phase-reduced network under heterogeneous intrinsic frequencies. 

\begin{thm}\label{thm static heterogeneous}
Consider the reduced networked phase dynamics
\begin{equation}\label{eq general phase 1}
\dot\theta_j
=
\omega_j
+
\sum_{k=1}^N A_{jk}\Gamma(\theta_k,\theta_j)
+
b(\theta_j)u_j,
~ j=1,\dots,N,
\end{equation}
where Hypothesis~\ref{ass reduction} holds. Let
$
R=\frac{1}{N}\sum_{k=1}^N e^{\mathrm{i}\theta_k},$
~$
Q=|R|^2,$
$
\frac{\partial Q}{\partial \theta_j}=\frac{2}{N}s_j,
$
with
$
s_j=\mathrm{Im}\!\bigl(Re^{-\mathrm{i}\theta_j}\bigr)
=\frac{1}{N}\sum_{k=1}^N \sin(\theta_k-\theta_j).
$
Assume that the intrinsic frequencies are heterogeneous but satisfy
$
\max_{1\le j\le N}\omega_j-\min_{1\le j\le N}\omega_j \le \Omega,
$
for some constant $\Omega>0$. Let $F_j(\theta)=\sum_{k=1}^N A_{jk}\,\Gamma(\theta_k,\theta_j)$, and
 there exists $M\ge 0$ such that 
\begin{equation}\label{eq:proj-bound}
\sum_{j=1}^N s_j\,F_j(\theta) \le  M\,\sum_{j=1}^N s_j^2.
\end{equation} 

Define the static feedback law
\begin{equation}\label{eq static law heterogeneous}
u_j
=
-\frac{\kappa s_j+\frac{\Omega}{2}\,\mathrm{sgn}(s_j)}{b(\theta_j)},
~
\kappa>M,
\end{equation}
Then, along the solutions of the controlled system~\eqref{eq general phase 1}--\eqref{eq static law heterogeneous}, we have
$
\lim_{t\to\infty}s_j(t)=0,~ j=1,\dots,N,
$ and
 the desynchronised state $Q=0$ is asymptotically stable.
\end{thm}

\begin{proof}
By the definition of $Q$, one has
\begin{equation}\label{eq Qdot est 0}
\dot Q
=
\sum_{j=1}^N \frac{\partial Q}{\partial \theta_j}\dot\theta_j
=
\frac{2}{N}\sum_{j=1}^N
s_j
\left(
\omega_j+\sum_{k=1}^N A_{jk}\Gamma(\theta_k,\theta_j)+b(\theta_j)u_j
\right),
\end{equation}
Using Eq.~\eqref{eq:proj-bound} we obtain

\begin{equation}\label{eq Qdot est 1}
\dot Q
\le
\frac{2}{N}\sum_{j=1}^N s_j\omega_j
+
\frac{2M}{N}\sum_{j=1}^N s_j^2
+
\frac{2}{N}\sum_{j=1}^N s_j\,b(\theta_j)u_j.
\end{equation}

Now let
\[
\bar\omega=\frac{\max_j\omega_j+\min_j\omega_j}{2},
\]
under the change of variables
\[
\tilde\theta_j=\theta_j-\bar\omega t,
\]
the order parameter $R$ and hence $Q=|R|^2$ remain unchanged. Thus, without loss of generality, one may assume that the shifted frequencies still denoted by $\omega_j$ satisfy
\[
|\omega_j|\le \frac{\Omega}{2},
~ j=1,\dots,N,
\]
it follows that
\[
\sum_{j=1}^N s_j\omega_j
\le
\sum_{j=1}^N |s_j|\,\frac{\Omega}{2},
\]
substituting the control law~\eqref{eq static law heterogeneous} into~\eqref{eq Qdot est 1} yields
\begin{equation}\label{eq Q decay heterogeneous}
\begin{aligned}
\dot Q
&\le
\frac{2}{N}\sum_{j=1}^N |s_j|\,\frac{\Omega}{2}
+
\frac{2M}{N}\sum_{j=1}^N s_j^2
-
\frac{2}{N}\sum_{j=1}^N s_j
\left(
\kappa s_j+\frac{\Omega}{2}\mathrm{sgn}(s_j)
\right) \\
&=
-\frac{2(\kappa-M)}{N}\sum_{j=1}^N s_j^2
+
\frac{2}{N}\sum_{j=1}^N |s_j|\,\frac{\Omega}{2}
-
\frac{2}{N}\sum_{j=1}^N |s_j|\,\frac{\Omega}{2} \\
&=
-\frac{2(\kappa-M)}{N}\sum_{j=1}^N s_j^2.
\end{aligned}
\end{equation}

Therefore, if $\kappa > M$, then $\dot Q \le 0$. Since $Q \in [0,1]$, $Q$ is nonincreasing and bounded from below. Moreover, if there exists some $s_j \neq 0$, then $Q$ is strictly decreasing.

The above inequality implies that only two cases can occur: 
\begin{itemize}
    \item[(1)] $Q \to 0$, in which case the system converges to the desynchronized state;
    \item[(2)] $\dot Q = 0$ only when $s_1=\cdots=s_N=0$. Geometrically, the condition $s_j=\mathrm{Im}(R\mathrm{e}^{-\mathrm{i}\theta_j})=0$ with $R=|R|\mathrm{e}^{-\mathrm{i}\psi}$ implies that there exists $\psi$ such that $\mathrm{e}^{-\mathrm{i}\theta_j}\in\{\mathrm{e}^{\mathrm{i}\psi},\mathrm{e}^{\mathrm{i}(\psi+\pi)}\}$, namely, the phases are distributed in at most two antipodal clusters.  In particular, although the fully synchronized state also satisfies $s_j=0$, in this case we have $\dot{\theta}_j=w_j$, the heterogeneous natural frequencies generally induce phase differences and thus destroy the complete synchronisation.
\end{itemize}
 
 Therefore, the system converges either to the desynchronized state $R=0$ or to a two-cluster synchronized state.
\end{proof}

\begin{rem}\label{rem GD heterogeneous}

\textcolor{black}{Equations~\eqref{eq static law heterogeneous} follows from interpreting the desynchronisation task as
a gradient-descent procedure on the synchronisation energy 
\(
    Q = |R|^{2}.
\)
However, since $Q(t)$ depends only on the phase variables $\theta_i(t)$ and has no explicit dependence on the control input $u(t)$, its gradient with respect to $Q(t)$ vanishes identically. To capture how the control influences the evolution of  $Q(t)$, it is therefore natural to consider the derivative $\dot{Q}(t)$ instead~\cite{fradkov2002speed,fradkov2007cybernetical,selivanov2012adaptive}, which is called speed gradient method.
According to the derivative of \(Q\) in Eq.~\eqref{eq Qdot est 0}, the directional derivative of \(\dot Q\) with respect to the control \(u_j\) is
\[
    \frac{\partial \dot Q}{\partial u_j}
    = \frac{2}{N}  b(\theta_j)s_j,
\]
a steepest--descent step in \(u_j\) therefore takes the form
\[
    u_j  \propto  -\, \frac{\partial \dot Q}{\partial u_j}
    = -\, \frac{2}{N} b(\theta_j) s_j,
\]
which yields the static law}
\[
u_j\sim -\frac{\kappa s_j}{b(\theta_j)}.
\]

In the heterogeneous setting, however, the term
\(
\sum_j s_j\omega_j
\)
appears in $\dot Q$ and may prevent monotone decay of $Q$. For this reason, the additional sign term
$
-\frac{\Omega}{2}\frac{\mathrm{sgn}(s_j)}{b(\theta_j)}
$
is introduced in~\eqref{eq static law heterogeneous} to compensate for the worst-case contribution of the frequency spread. Hence~\eqref{eq static law heterogeneous} may be viewed as a speed-gradient feedback corrected by a robustness term accounting for the heterogeneity of the intrinsic frequencies.
\end{rem}

\begin{rem}\label{rem dynamic meanfield heterogeneous}
Besides the static law~\eqref{eq static law heterogeneous}, we could introduce dynamic laws in the same spirit. Motivated by filtering the instantaneous speed-gradient update through a first-order leakage, we consider
\begin{equation}\label{eq dynamic law heterogeneous}
\dot u_j
=
-\rho u_j
-
\kappa s_j,~
\rho>0,~\kappa>M.
\end{equation}

If the actuator gain is retained explicitly, the corresponding controller is 
\begin{equation}\label{eq dynamic law heterogeneous b}
\dot u_j
=
-\rho u_j
-
\kappa b(\theta_j)s_j
.
\end{equation}
For simplicity, in the following text, we mainly consider dynamic control in Eq.~\eqref{eq dynamic law heterogeneous}.

\end{rem}

}

\begin{rem}\label{rem mean field}
The distributed additive control in system~(3.1) may be impractical in settings
with low-rank actuation, where only a small number of input channels are available.
A common rank-one arrangement in Kuramoto networks is a mean-field control that
modulates both sine and cosine couplings~\cite{wang2023desynchronizing}:
\begin{equation}
    \label{eq:mf-kuramoto}
    \dot\theta_j
    = \omega_j
      + K s_j
      + C \, s_j
      + S \, p_j,
    ~ j = 1,\dots,N,
\end{equation}
where $C,S\in\mathbb{R}$ are scalar control gains shared by all oscillators, and
\[
    s_j
    = \frac{1}{N}\sum_{k=1}^N \sin(\theta_k - \theta_j),
    ~
    p_j
    = \frac{1}{N}\sum_{k=1}^N \cos(\theta_k - \theta_j),
\]
denote the local mean fields of sine and cosine interactions, respectively.

For the general oscillator network~(2.1), the mean-field control formulation is
\begin{equation}\label{eq main mean field}
\dot{\vx}_j=\vf(\vx_j)+\sum_{k=1}^NA_{jk}\vh(\vx_k,\vx_j)+\vg(\vx_j)\left(\frac{C}{N}\sum_{k=1}^N\sin(\theta_k-\theta_j)+\frac{S}{N}\sum_{k=1}^N\cos(\theta_k-\theta_j)\right),
\end{equation}
with scalar input actuator $\vg\in\mathbb{R}^{n\times1}$.
The reduced phase dynamics under
Hypothesis~\ref{ass reduction} takes the form
\begin{equation}
    \label{eq:mf-general-phase}
    \dot\theta_j
    = \omega
      + F_j(\theta)
      + b(\theta_j)\Bigl( C\, s_j + S\, p_j \Bigr),
    ~ j=1,\dots,N,
\end{equation}
where $F_j(\theta)=\sum_{k=1}^N A_{jk}\Gamma(\theta_k,\theta_j)$,
representing a rank-one actuation of the whole network.

As in Theorem~\ref{thm static heterogeneous}, the synchronisation energy $Q=|R|^2$ satisfies
\begin{equation*}
    \dot Q
    = \frac{2}{N}\sum_{j=1}^N s_j F_j(\theta)
      + C \sigma_1^b + S \sigma_2^b,
\end{equation*}
where
\begin{equation*}
    \sigma_1^b
    = \frac{2}{N}\sum_{j=1}^N b(\theta_j) s_j^2,
    ~
    \sigma_2^b
    = \frac{2}{N}\sum_{j=1}^N b(\theta_j) s_j p_j .
\end{equation*}

Due to the nondegenerate actuator gain
$b_{\min}\le b(\theta)\le b_{\max}$, it follows that
$
    \dot Q
    \le \frac{2M}{N}\sum_{j=1}^N s_j^2
       + C\sigma_1^b + S\sigma_2^b.
$
A simple mean-field static control is obtained by fixing $S\equiv 0$ and choosing
$C<0$ sufficiently large in magnitude.  
For instance, if $
        C < -\frac{M}{b_{\min}},$
then
$
    \dot Q
    \le -\frac{2|C|b_{\min}-2M}{N}\sum_{j=1}^N s_j^2 \le 0,
$
and $\dot Q=0$ only when $s_j\equiv 0$.
Hence the static mean-field feedback recovers the desynchronisation result of
Theorem~\ref{thm static heterogeneous} as a special case.

To obtain a fully dynamic and adaptive mean-field feedback control law, coefficients $C$ and $S$ must evolve
according to the gradients $\partial\dot Q/\partial C$ and
$\partial\dot Q/\partial S$, which necessarily involve the actuator gain
$b(\theta_j)$.
We devise the detailed dynamic formulation for the mean-field control as follows,
\begin{equation*}
\dot C(t)
= -\rho C(t)
- \kappa\sigma_1^b(t),
~
\dot S(t)
= -\rho S(t)
- \kappa\sigma_2^b(t),
\end{equation*}
with $\rho>0$, $\kappa>M$,~ and $M$ coming from Eq.~\eqref{eq:proj-bound}. A simpler alternative dynamic mean-field control law is given by
\begin{equation}\label{eq dynamic mean field final}
\dot C(t)
= -\rho C(t)
- \kappa\sigma_1(t),
~
\dot S(t)
= -\rho S(t)
- \kappa\sigma_2(t),
\end{equation}
where
$
    \sigma_1
    = \frac{2}{N}\sum_{j=1}^N  s_j^2,
    ~
    \sigma_2
    = \frac{2}{N}\sum_{j=1}^N s_j p_j ,
$
and in the following text, we mainly consider the dynamic mean-field control in Eq.~\eqref{eq dynamic mean field final}.

\end{rem}

\begin{rem}
Eq.~\eqref{eq:proj-bound} is employed to bound the nonlinear phase response term in the controlled vector field of $Q$. For example, for the fully connected network $A_{jk}=1/N$ with $\Gamma(\phi)=\sin\phi$ (the classical Kuramoto case), $
F_j(\theta)=\frac{1}{N}\sum_{k=1}^N \sin(\theta_k-\theta_j)=s_j,$
hence $\sum_j s_j F_j(\theta)=\sum_j s_j^2$, i.e., \eqref{eq:proj-bound} holds with \textbf{$M=1$}. Furthermore,
if~$\Gamma$ is odd and satisfies $|\Gamma(\phi)|\le L\,|\sin\phi|$ for some $L>0$, then
\eqref{eq:proj-bound} holds with $M\le L\,\max_j\sum_{k}A_{jk}$. 

To further validate the soundness of Eq.~\eqref{eq:proj-bound}, we provide numerical test in Section~\ref{sec exp}.
\end{rem}

\section{Desynchronisation with Event-triggered Control}\label{sec etc}

In the previous section, we establish the stabilization analysis for desynchronisation control designed from the order parameter. However, the proposed control laws depend on the continuous update of the state, rendering large communicational cost in real-world applications. To circumvent such issue, we introduce the event-triggered mechanism into the proposed control laws to reduce the update times in the control process. For the physical interpretation and the algorithmic coherence, we consider to design the event function with the order parameter. We begin with a simplified situation where the heterogeneous frequencies degenerate to homogeneous frequencies $\omega_j\equiv\omega$.

\begin{thm}\label{thm general system etc homo}
\textcolor{black}{For the collective dynamics of coupled oscillators~\eqref{eq main} under Hypothesis~\ref{ass reduction} and condition~\eqref{eq:proj-bound}, and the controllers devised in Eqs.~\eqref{eq static law heterogeneous},\eqref{eq dynamic law heterogeneous b}, we introduce the event-triggered mechanisms as
\begin{equation}\label{eq event}
    \begin{aligned}
        t_{k+1}=\inf\{t>t_k:|Q(t)-Q(t_k)|-\delta Q(t)=0\},\\
    \end{aligned}
\end{equation}
here $\delta>0$ is a predefined hyperparameter. 
Then the controlled system under event-triggered mechanism avoids Zeno behaviour, i.e., there always exist some constant $T>0$ such that $t_{k+1}-t_k\ge T$ for any consecutive trigger times $t_k$ and $t_{k+1}$. Specifically, for the static law $u_j(t)  = -\kappa \frac{s_j(t_k)}{b_{\max}},~t\in[t_k,t_{k+1})$, the lower bound is $T=\dfrac{1}{2M+\kappa}\log\frac{1+\delta}{1+\frac{\kappa\delta}{2(M+\kappa)}};$ for the dynamic law $u_j(t)=-\frac{\kappa}{b_{\max}}\int_{t_{k-1}}^{t_k}e^{-\rho(t_k-\tau)}s_j(\tau)\mathrm{d}\tau$,~$t\in[t_k,t_{k+1})$, the lower bound is $T=\dfrac{1}{2M+2\frac{\kappa}{\rho(1-\delta)}}\log\frac{1+\delta}{1+\delta\frac{\kappa(1+\delta)}{2(M\rho(1-\delta)+\kappa)}};$
for the dynamic mean-field law \\$C(t)=-\frac{\kappa}{b_{\max}}\int_{t_{k-1}}^{t_k}e^{-\rho(t_k-\tau)}\sigma_1(\tau)\mathrm{d}\tau,$$~$$S(t)=-\frac{\kappa}{b_{\max}}\int_{t_{k-1}}^{t_k}e^{-\rho(t_k-\tau)}\sigma_2(\tau)\mathrm{d}\tau$,~$t\in[t_k,t_{k+1}),$ the lower bound is $T=\dfrac{\log(1+\delta)}{2(M+\kappa/\rho)}.$ }

\end{thm}

\begin{proof}

From the event condition~\eqref{eq event}, the triggering occurs when $\frac{|Q(t)-Q(t_k)|}{Q(t)}$ goes from $0$ to $\delta$. Therefore, we come to deduce the dynamics of $\frac{|Q(t)-Q(t_k)|}{Q(t)}$. To simplify the analysis, we introduce the error state $e_Q(t)=Q(t_k)-Q(t)$ with the update law $e_Q(t_k)=0$. Using Eq.~\eqref{eq Qdot est 1} we have 
\begin{equation}\label{eq homo etc proof 0}
\begin{aligned}
    \dfrac{\mathrm{d}}{\mathrm{d}t}\dfrac{|e_Q|}{Q}&=\dfrac{\mathrm{d}}{\mathrm{d}t}\dfrac{(e_Q^\top e_Q)^{1/2}}{(Q^\top Q)^{1/2}}\\
    &=\dfrac{\frac{1}{2}(e_Q^\top e_Q)^{-1/2}2e_Q^\top\dot{e}_Q(Q^\top Q)^{1/2}}{Q^\top Q}\\
    &-\dfrac{\frac{1}{2}(Q^\top Q)^{-1/2}2Q^\top\dot{Q}(e_Q^\top e_Q)^{1/2}}{Q^\top Q}\\
    &=\dfrac{e_Q^\top\dot{e}_Q}{|e_Q|Q}-\dfrac{Q^\top\dot{Q}}{QQ}\dfrac{|e_Q|}{Q}\\
    &=-\dfrac{e_Q^\top\dot{Q}}{|e_Q|Q}-\dfrac{Q^\top\dot{Q}}{QQ}\dfrac{|e_Q|}{Q}\\
    &\le\dfrac{|e_Q||\dot{Q}|}{|e_Q|Q}+\dfrac{Q|\dot{Q}|}{QQ}\dfrac{|e_Q|}{Q}\\
    &=\dfrac{|\dot{Q}|}{Q}\left(1+\dfrac{|e_Q|}{Q}\right)\\
    &\le\dfrac{\Big|\frac{2}{N}\sum_{j=1}^N s_j\Big(\omega_j+Ms_j+b(\theta_j)u_j(t_k)\Big)\Big|}{Q}\left(1+\dfrac{|e_Q|}{Q}\right).
\end{aligned}
\end{equation}

Similarly to the proof in Theorem~\ref{thm static heterogeneous}, we can remove the common frequency $\omega_j$ here by variable substitution. Then, for the static law $u_j(t_k)=-\kappa s_j(t_k)$ we have

\begin{equation}\label{eq kura etc proof 1}
\dfrac{\mathrm{d}}{\mathrm{d}t}\dfrac{|e_Q|}{Q}\le     \dfrac{\Big|\frac{2}{N}\sum_{j=1}^N s_j\Big(Ms_j-\kappa b(\theta_j)s_j(t_k)/b_{\max}\Big)\Big|}{Q}\left(1+\dfrac{|e_Q|}{Q}\right).
\end{equation}

Since $s_j=\text{Im}(Re^{-\text{i}\theta_j})\le|R|$ and $b(\theta_j)/b_{\max}\le1$ we have

\begin{equation*}
\begin{aligned}
\dfrac{\mathrm{d}}{\mathrm{d}t}\dfrac{|e_Q|}{Q}&\le     \dfrac{\Big|\frac{2}{N}\sum_{j=1}^N \Big(MQ+\kappa\frac{Q+Q(t_k)}{2}\Big)\Big|}{Q}\left(1+\dfrac{|e_Q|}{Q}\right)\\
&=    \dfrac{\Big|\frac{2}{N}\sum_{j=1}^N \Big(MQ+\kappa\frac{Q+Q+e_Q}{2}\Big)\Big|}{Q}\left(1+\dfrac{|e_Q|}{Q}\right)\\
&=\dfrac{\Big|\Big(2(M+\kappa)Q+\kappa e_Q\Big)\Big|}{Q}\left(1+\dfrac{|e_Q|}{Q}\right)\\
&\le \left(2(M+\kappa)+\kappa\dfrac{|e_Q|}{Q}\right)\left(1+\dfrac{|e_Q|}{Q}\right).
\end{aligned}
\end{equation*}

By comparison principle, we have the triggering time of $\frac{|e_Q|}{Q}$ happens after the variable  $z$ increases from $0$ to $\delta$, where the dynamic of $z$ is 
\begin{equation*}
\begin{aligned}
      \dot{z}& = \left(2(M+\kappa)+\kappa z\right)\left(1+z\right)\\
      z_0&=0,~z_T= \delta.
\end{aligned}
\end{equation*}

We have $
    \frac{\mathrm{d}z}{(1+az)(1+z)}=b\mathrm{d}t,$
where $a=\frac{\kappa}{2(M+\kappa)}$, $b=2(M+\kappa)$.
Then we have 
\begin{equation*}
\begin{aligned}
    \dfrac{\mathrm{d}z}{(1+az)(1+z)}&=\dfrac{a}{a-1}\left(\dfrac{1}{1+az}-\dfrac{1}{a(1+z)}\right)\mathrm{d}z\\
    &=\dfrac{1}{a-1}\left(\mathrm{d}\log(1+az)-\mathrm{d}\log(1+z)\right)=b\mathrm{d}t.
\end{aligned}
\end{equation*}

By integrating the above equation, the time $T$ when $z$ achieves $\delta$ satisfies,
\begin{equation*}
    \begin{aligned}
        &\dfrac{1}{a-1}\left(\log(1+a\delta)-\log(1+\delta)\right)=bT\\
        &\to T=\dfrac{1}{b(a-1)}\log\left(\dfrac{1+a\delta}{1+\delta}\right)\\
        &=\dfrac{1}{b(1-a)}\log\left(\dfrac{1+\delta}{1+a\delta}\right)\\
        &=\dfrac{1}{2M+\kappa}\log\frac{1+\delta}{1+\frac{\kappa\delta}{2(M+\kappa)}}.
    \end{aligned}
\end{equation*}

For the dynamic law, the controller $u_j(t_k)$ implemented to the system is calculated by the differential equation
$
 \dot u_j(t) = - \rho u_j(t)-\kappa  s_j(t),~t\in[t_{k-1},t_k),$
whose initial value is reset as $0$ from the last controller $u_j(t_{k-1})$.
Then we have 
\begin{equation*}
\begin{aligned}
 u_j(t_k)=-\kappa\int_{t_{k-1}}^{t_k}e^{-\rho(t_k-\tau)}s_j(\tau)\mathrm{d}\tau.
\end{aligned}
\end{equation*}

To control the term $b(\theta_j)$ in Eq.~\eqref{eq homo etc proof 0}, we add the multiplicative factor $1/b_{\max}$ to the above controller. 

Therefore, for $s_j(t)b(\theta_j)u_j(t_k)$ we have 
\begin{equation}\label{eq kura etc proof 2}
\begin{aligned}
    s_j(t)b(\theta_j)u_j(t_k)&\le \kappa\frac{b(\theta_j)}{b_{\max}}\int_{t_{k-1}}^{t_k}e^{-\rho(t_k-\tau)}|s_j(\tau)||s_j(t)|\mathrm{d}\tau\\
    &\le \kappa\int_{t_{k-1}}^{t_k}e^{-\rho(t_k-\tau)}|R(\tau)||R(t)|\mathrm{d}\tau\\
    &\le \kappa\int_{t_{k-1}}^{t_k}e^{-\rho(t_k-\tau)}\frac{Q(\tau)+Q(t)}{2}\mathrm{d}\tau\\
\end{aligned}
\end{equation}

According to the event condition we know that 

\begin{equation*}
    \begin{aligned}
        |Q(t_k)-Q(t_{k-1})|&=\delta Q(t_k),\\
Q(t)-Q(t_{k-1})|&\le\delta Q(t),~t\in[t_{k-1},t_{k}),
    \end{aligned}
\end{equation*}
which implies,
\begin{equation*}
    \begin{aligned}
        Q(t_{k-1})&\le(1+\delta) Q(t_k),\\
Q(t)&\le\frac{1}{1-\delta} Q(t)\le\frac{1+\delta}{1-\delta}Q(t_k),~t\in[t_{k-1},t_{k}).
    \end{aligned}
\end{equation*}

Substituting the above results into Eq.~\eqref{eq kura etc proof 2} we obtain,
\begin{equation}\label{eq general homo dynamic 2}
\begin{aligned}
   s_j(t)b(\theta_j)u_j(t_k)
    &\le \kappa\int_{t_{k-1}}^{t_k}e^{-\rho(t_k-\tau)}\frac{\frac{1+\delta}{1-\delta}Q(t_k)+Q(t)}{2}\mathrm{d}\tau\\
    &= \kappa\int_{t_{k-1}}^{t_k}e^{-\rho(t_k-\tau)}\frac{\frac{1+\delta}{1-\delta}(e_Q+Q(t))+Q(t)}{2}\mathrm{d}\tau\\
     &\le \frac{\kappa}{2\rho}\left(\frac{1+\delta}{1-\delta}|e_Q|+\frac{2}{1-\delta}Q\right).\\    
\end{aligned}
\end{equation}

Substituting the above inequality into $\dot{Q}$ like Eq.~\eqref{eq kura etc proof 1} we have,

\begin{equation*}
\begin{aligned}
 \dfrac{\mathrm{d}}{\mathrm{d}t}\dfrac{|e_Q|}{Q}
 &\le     \dfrac{\Big|\Big(2MQ+\frac{\kappa}{\rho}\left(\frac{1+\delta}{1-\delta}|e_Q|+\frac{2}{1-\delta}Q\right)\Big)\Big|}{Q}\left(1+\dfrac{|e_Q|}{Q}\right)\\
 &=\left(2M+\frac{2\kappa}{\rho(1-\delta)}+\frac{\kappa(1+\delta)}{\rho(1-\delta)}\frac{|e_Q|}{Q}\right)\left(1+\dfrac{|e_Q|}{Q}\right).
\end{aligned}
\end{equation*}

By comparison principle we obtain the lower bound of the time that $\frac{|e_Q|}{Q}$ goes from $0$ to $\delta$ as
\begin{equation*}
    t_{k+1}-t_k\ge \dfrac{1}{2M+2\frac{\kappa}{\rho(1-\delta)}}\log\frac{1+\delta}{1+\delta\frac{\kappa(1+\delta)}{2(M\rho(1-\delta)+\kappa)}},~\forall k.
\end{equation*}

For dynamic mean-field control, similarly we have 
\begin{equation*}
\dfrac{\mathrm{d}}{\mathrm{d}t}\dfrac{|e_Q|}{Q}\le     \dfrac{\Big|2MQ+C(t_k)\sigma^b_1+S(t_k)\sigma_2^b\Big|}{Q}\left(1+\dfrac{|e_Q|}{Q}\right).
\end{equation*}

We note that 
\begin{equation*}
    \begin{aligned}
\sigma_1&=\frac{1}{N}\sum_{j=1}^Ns_j^2\le Q,\\
 \sigma_1^b&=\frac{1}{N}\sum_{j=1}^Nb(\theta_j)s_j^2\le b_{\max}Q,\\
\sigma_2&=\frac{1}{N}\sum_{j=1}^Ns_jp_j\le \frac{1}{N}\sum_{j=1}^N(s_j^2+p_j^2)/2\le Q,\\
\sigma_2^b&=\frac{1}{N}\sum_{j=1}^Nb(\theta_j)s_jp_j\le \frac{1}{N}\sum_{j=1}^Nb(\theta_j)(s_j^2+p_j^2)/2\le b_{\max}Q,        
    \end{aligned}
\end{equation*}

combining with $Q\le 1$, we have 
\begin{equation*}
\begin{aligned}
   |C(t_k)\sigma_1^b(t)|
       &\le b_{\max}Q(t) \frac{\kappa}{b_{\max}}\int_{t_{k-1}}^{t_k}e^{-\rho(t_k-\tau)}Q(\tau)\mathrm{d}\tau\\
    &\le \kappa Q(t)\int_{t_{k-1}}^{t_k}e^{-\rho(t_k-\tau)}\mathrm{d}\tau\\
     &\le \frac{\kappa Q(t)}{\rho},\\    
\end{aligned}
\end{equation*}
similarly we have $|S(t_k)\sigma_2^b(t)|\le\frac{\kappa Q(t)}{\rho}$, which implies that
\begin{equation*}
    \dfrac{\mathrm{d}}{\mathrm{d}t}\dfrac{|e_Q|}{Q} \le     \left(2M+2\frac{\kappa}{\rho}\right)\left(1+\dfrac{|e_Q|}{Q}\right),
\end{equation*}
by comparison principle we obtain the lower bound of the inter-event time as
$
    T = \frac{\log(1+\delta)}{2(M+\kappa/\rho)},$
 which completes the proof.
\end{proof}

\textcolor{black}{To proceed, we extend the results in Theorem~\ref{thm general system etc homo} to the heterogeneous networked dynamics, although this extension is not straightforward. The main difficulty is that, in the heterogeneous case, the proposed controller~\eqref{eq static law heterogeneous} involves the instantaneous sign $\mathrm{sgn}(s_j)$, whereas under an event-triggered implementation the feedback is naturally frozen between two consecutive triggering instants.} 

\textcolor{black}{To circumvent the challenge, one seemingly direct approach is to subtract $\omega_j$ in each control channel so as to cancel the heterogeneous frequencies explicitly. However, this requires exact knowledge of all intrinsic frequencies and introduces a non-vanishing oscillator-dependent bias into the control input, which is not desirable in practice. Another possible approach is to freeze the sign term at $t_k$, namely to replace $\mathrm{sgn}(s_j(t))$ by $\mathrm{sgn}(s_j(t_k))$. Yet, in order to recover the key estimate for the lower bound of $t_{k+1}-t_k$, one would need $s_j(t)$ and $s_j(t_k)$ to retain the same sign for all $t\in [t_k,t_{k+1}]$. Such a requirement is generally impossible to verify a priori, before simulating the trajectories, and is therefore overly restrictive. For this reason, we adopt in Theorem~\ref{thm general system etc heter} a \textit{hybrid event-triggered control}: the gradient-descent component is updated only at the triggering instants, while the robustness term retains the instantaneous sign information, thereby yielding a tractable extension of the event-triggered scheme to the heterogeneous case.}

\begin{thm}\label{thm general system etc heter}
We consider the collective dynamics of coupled oscillators~\eqref{eq main} under Hypothesis~\ref{ass reduction}, condition~\eqref{eq:proj-bound}, and the event mechanism in Eq.~\eqref{eq event}, where the heterogeneous frequencies are bounded by $
\max_{1\le j\le N}\omega_j-\min_{1\le j\le N}\omega_j \le \Omega,
$ as stated in Theorem~\ref{thm static heterogeneous}. The hybrid event-triggered controllers over time interval $[t_k,t_{k+1})$ are defined as, 
\begin{equation}\label{eq general control etc 1}
    \text{(Static)}~u_j(t)  = -\kappa \frac{s_j(t_k)-\mathrm{sgn}(s_j)\frac{\Omega}{2}}{b_{\max}},
\end{equation}
\begin{equation}\label{eq general control 2}
    \text{(Dynamic)}~ u_j(t)  = -\frac{1}{b_{\max}}\left(\kappa\int_{t_{k-1}}^{t_k}\mathrm{e}^{-\rho(t_k-\tau)} s_j(\tau)\mathrm{d}\tau-\mathrm{sgn}(s_j)\frac{\Omega}{2}\right),
\end{equation}

   Then the corresponding controlled system~\eqref{eq main} avoids Zeno behaviour, in which case the minimal inter-event time converges to zero~\cite{tabuada2007event}.
\end{thm}

\textcolor{black}{The proof of Theorem~\ref{thm general system etc heter} is analogous to that of Theorem~\ref{thm general system etc homo} and is therefore omitted. The only essential difference is that, in the heterogeneous setting, after subtracting the midpoint frequency one has $|\omega_j|\le \Omega/2$, which yields an extra term $\sum_{j=1}^N s_j\omega_j$ in  $\dot Q$ and hence in the evolution of $|e_Q|/Q$. The additional sign term in the hybrid controllers \eqref{eq general control etc 1}-\eqref{eq general control 2} is introduced exactly to cancel this worst-case frequency contribution. After this compensation, the rest of the argument is identical to that of Theorem~\ref{thm general system etc homo}, leading to the same type of comparison system and thus a strictly positive lower bound for $t_{k+1}-t_k$. We do not yet have an analogous theoretical result for the dynamic mean-field control, since a common control input cannot separately compensate for the heterogeneous intrinsic frequencies. Nevertheless, the numerical experiments still demonstrate the effectiveness of the dynamic mean-field control under the event-triggered mechanism.
}

\section{Pseudo phase reduction}\label{sec pseudo}

\begin{figure}
	\centering
    \includegraphics[width=0.8\textwidth]{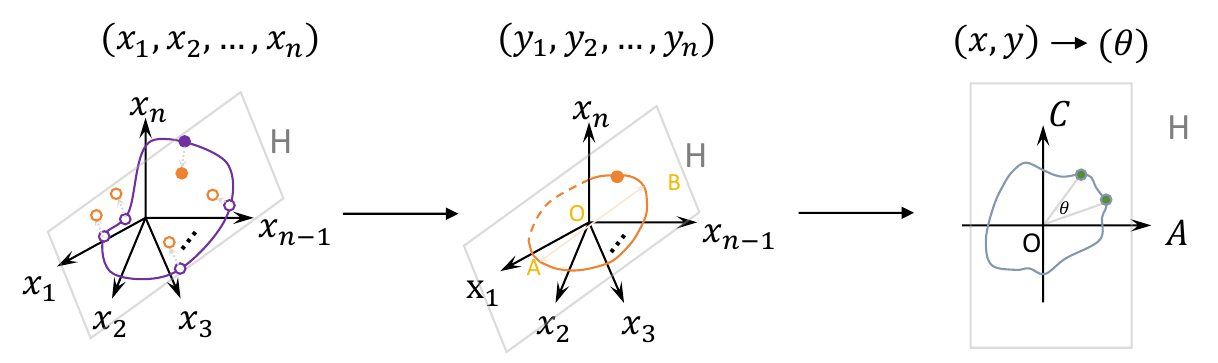}
    \caption{Illustration of the pseudo phase approach. The high-dimensional limit cycle is first projected to a hyperplane. Then we transform the projected cycle in high-dimensional hyperplane to a 2-dimensional plane by coordinate transformation. Finally we obtain the pseudo phase as the geometric phase of the 2-dimensional cycle.
     }
\vskip -0.15in
\end{figure}
While in the previous sections we established a fundamental route to desynchronise the coupled oscillations based on their phase dynamics, exact reduced phase dynamics are generally intractable in networked high-dimensional oscillators since most oscillations lack analytical expressions.  This motivates a robust, model-agnostic surrogate constructed directly from available state measurements over limit-cycles.

\begin{ass}
For the limit cycle $\Gamma\subset\mathbb{R}^d$ of the isolated oscillator dynamics $\dot{\vx}=\vf(\vx)$, we assume a dense sample set $\mathcal{X}=\{x_k\}_{k=1}^N\subset\Gamma$ is available.
\end{ass}
Our aim is to find the hyperplane $\mathbb{H}$ that is the closest to the limit cycle in the sense that the projection of the limit cycle to $\mathbb{H}$ keeps the similarity to the limit cycle as much as possible. To find such a hyperplane, we need two orthogonal vectors as horizontal and vertical axis along with a pivot point over the limit cycle. We first construct the pivot point and the horizontal axis as follows.
\begin{definition}\label{def:AB-v}
For the dense sample set $\mathcal{X}$, we define the antipodal pair of points $(A,B)$ as,
\begin{equation}\label{eq:AB}
    (A,B)\ \in\ \arg\max_{x_i,x_j\in\mathcal{X}}\|x_i-x_j\|,
\end{equation}
\end{definition}
then we further define the horizontal axis direction as $
    \vv=\frac{B-A}{\|B-A\|}\in \mathbb{S}^{n-1},$
we can select $A$ as a pivot point and $\vv$ as horizontal axis. 

To proceed, we determine the vertical axis with the Principle Component Analysis (PCA) principle. Consider planes that pass through the line $\overline{AB}$ and rotate within the family $\mathrm{span}\{v,u\}$, where $u\in v^\perp$ is a unit vector. The squared distance from $x_k$ to the plane $\Pi(v,u):=\{A+sv+tu: s,t\in\mathbb{R}\}$ equals $\|(I - vv^\top - uu^\top)(x_k-A)\|^2$. The total residual over $\mathcal{X}$ is
\begin{equation}\label{eq:residual}
\mathcal{J}(u) = \sum_{k=1}^N \|(I - vv^\top - uu^\top)(x_k-A)\|^2
 = \sum_{k=1}^N\|p_k\|^2 - \sum_{k=1}^N (u^\top p_k)^2.
\end{equation}

Let $\Sigma := \frac{1}{N}\sum_{k=1}^N p_k p_k^\top$ be the covariance on $v^\perp$. Since the first term of \eqref{eq:residual} is constant w.r.t.\ $u$, minimizing $\mathcal{J}(u)$ is equivalent to maximizing $u^\top \Sigma u$. Then we have the following theorem.
\begin{thm}
The plane $\Pi(\vv,\vp)$ that minimizes the total squared distance \eqref{eq:residual} is uniquely determined by $v$ and $\vp$, where 
$
\vp \in \arg\max_{\|\vu\|=1,  \vu\perp \vv} \vu^\top\Sigma \vu,
$
i.e., $\vp$ is the principal eigenvector of $\Sigma$ on $\vv^\perp$, then we define $\vp$ as the vertical axis associated with pivot $A$ and horizontal axis $\vv$.
\end{thm}

\noindent\emph{Proof.} For a unit $u\in v^\perp$, the residual is 
$\sum_k\|(I-vv^\top-uu^\top)(x_k-A)\|^2
=\sum_k\|p_k\|^2-\sum_k(u^\top p_k)^2$,
i.e., a constant minus a Rayleigh quotient. Maximising $u^\top\Sigma u$ yields $u_\star$. \hfill$\square$

Thus, the best-fitting plane is spanned by the chord direction $v$ and the dominant variance direction $u_\star$ of the projected cycle. Based on the aforementioned axes and pivot point, we obtain the pseudo phase as an approximation to the intractable phase as
\begin{definition}\label{def:pseudo-theta}
With the orthonormal basis $\{\vv,\vp\}$, define the relative coordinates $(y_1(\vx),y_2(\vx))$ for any $\vx\in\mathbb{R}^n$ as 
\begin{equation}\label{eq:planar-coords}
    y_1(\vx)=(\vx-A)^\top \vv, 
    y_2(\vx)=(\vx-A)^\top \vp , 
    \end{equation}
and the pseudo phase
\begin{equation}\label{eq:pseudo-theta}
    \tilde\theta(x) =\text{arg} \left(y_1(x)+\text{i}  y_2(x)\right)\in(-\pi,\pi].
\end{equation}
\end{definition}

In practice, the argument in Eq.~\eqref{eq:pseudo-theta} is calculated by the computing package \texttt{atan2} in \texttt{Python} or \texttt{Matlab}.

Now we can implement the proposed desynchronisation control laws under the event-triggered mechanism by replacing the phase $\theta$ by the pseudo phase $\tilde{\theta}$,
\begin{equation}\label{eq:order-pseudo}
    \tilde r =\ \frac{1}{N}\sum_{i=1}^N e^{\mathrm{i}\tilde\theta(x_i)},~
    Q(x) =\ |\tilde r|^2.
\end{equation}

\section{Numerical Experiments}\label{sec exp}
 In this section, several representative examples are provided to validate the usefulness of the proposed desynchronisation control laws and the obtained theoretical results. Moreover, we discuss the application scenarios and limitations of the proposed framework with a slow-fast dynamical system. 
   
    \subsection{Coupled Van der Pol oscillators} \label{sub sec vdp}
    To demonstrate the effectiveness of the proposed control strategies, we first apply both the static and dynamic control laws to a network of coupled Van der Pol (VDP) oscillators. The coupled VDP system is expressed as
 \begin{equation}\label{eq vdp main}
 \begin{aligned}
        \dot{v}_i&= w_i,\\
        \dot{w}_i &= \mu_i(1-v_i^2)w_i-v_i+\frac{K}{N}\sum_{j=1}^NA_{ij}\frac{v_j-v_c}{1+e^{-(v-v_0)/v_{th}}}+u_i,~i=1,...,N,
 \end{aligned}    
 \end{equation}
where $v_i, w_i$ denote the position and velocity of the $i$th oscillator, $\mu>0$ is the strength of damping force, $K=0.1$ is the coupling strength, $\vA = (A_{ij})$ represents the network adjacency matrix, and $u_i(t)$ denotes the external control input. In the absence of control, the system exhibits self-sustained oscillations, and strong coupling leads to a synchronous limit cycle. For the network structure, we consider three synthetic networks: Watts-Strogatz network, \text{Erd\H{o}s--R\'enyi} network, and scale free network, all these networks are generated by the \texttt{Networkx} package in \texttt{Python}. In addition, we consider two realistic brain networks: Macaque cortical connectivity network within one hemisphere~\cite{kaiser2006nonoptimal}, and C. elegans local network of 131 frontal neurons~\cite{kotter2004online}.

The objective is to realize \emph{desynchronization} of the oscillators through control laws proposed in Section~\ref{sec desyn}. Before presenting the performance of the proposed controllers, we numerically examine the validity of the bound in Eq.~\eqref{eq:proj-bound}. Specifically, we evaluate the ratio $|\sum_{j=1}^Ns_jF_j(\theta)|/\sum_{j=1}^Ns_j^2$ over a collection of sampled phase configurations. Uniform boundedness of this ratio provides numerical evidence supporting Eq.~\eqref{eq:proj-bound}. We consider two network topologies and two coupling mechanisms: the nonlinear coupling introduced in Eq.~\eqref{eq vdp main}, and the linear coupling term $K/N\sum_{j=1}^NA_{ij}v_j $, applied to $w_i$-dynamics. As shown in Figure~\ref{fig_vdp_validate}, the computed ratios remain uniformly bounded across all sampled configurations for both network topologies and coupling mechanisms. These results provide numerical support for the applicability of Eq.~\eqref{eq:proj-bound} to the systems considered below.

\begin{figure}\label{fig_vdp_validate}
	\centering
    \includegraphics[width=1.0\textwidth]{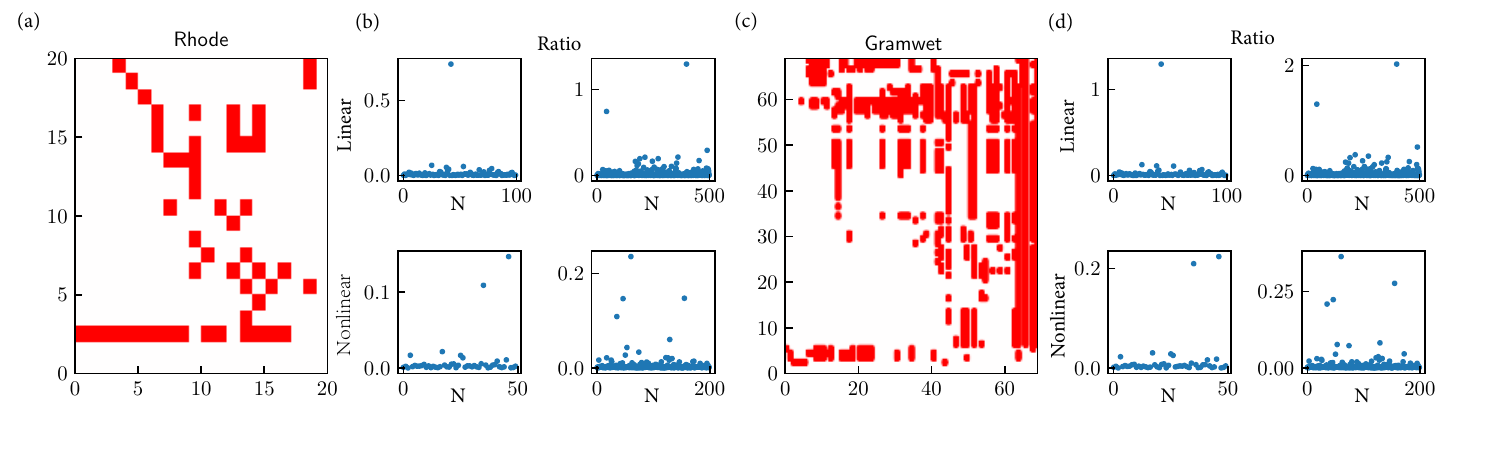}
	\vskip -0.17in
    \caption{Validation of Eq.~\eqref{eq:proj-bound} using Van der Pol oscillators. The two heatmaps (a), (c) correspond to the adjacency matrix of two real network structures: Rhode network and Gramwet network~\cite{tian2017articulation}, respectively. Red color represents the edge. The small panels in (b),(d) show the sampled ratio $|\sum_j s_jF_j|/\sum_j s_j^2$ under different coupling mechanisms. A numerical indication that Eq.~\eqref{eq:proj-bound} holds is that this ratio remains bounded and does not diverge to infinity. The first row corresponds to linear coupling, whereas the second row corresponds to nonlinear coupling in Eq.~\eqref{eq vdp main}. The horizontal axis in each small panel denotes the sample size.
     }
\vskip -0.15in
\end{figure}

We consider the homogeneous oscillations where different oscillators share the same strength of the damping $\mu_i=\mu=0.1$. The network is fixed as a Watts-Strogatz network with $N=50$.  As shown in the Figs.~\ref{fig_vdp1a},\ref{fig_vdp1b}, the pseudo phase reflects the regular and smoothly varying oscillation over the limit cycle. We respectively apply the distributed control $u_i$ and the mean-field control $u_i=\frac{u}{N}\sum_{j=1}^N\sin(\theta_j-\theta_i)$  to the controlled dynamics~\eqref{eq vdp main}, the results in figure~\ref{fig_vdp1c} show that the static and mean-field control laws quickly reduce the order parameter after the control signals are applied, while the dynamic laws need a period of warming up until the order parameter begins to decrease. The reason is that dynamic control initiated from zero value $u_i=0$, and needs some time to exceed some threshold to regulate the synchronisation state. Furthermore, we illustrate the desynchronisation effects under different control laws in Figs.~\ref{fig_vdp1d}-\ref{fig_vdp1f}.

\begin{figure}\label{fig_vdp1}
	\centering
    \subfigure{\label{fig_vdp1a}}
    \subfigure{\label{fig_vdp1b}}
    \subfigure{\label{fig_vdp1c}}
        \subfigure{\label{fig_vdp1d}}
    \subfigure{\label{fig_vdp1e}}
    \subfigure{\label{fig_vdp1f}}
    \includegraphics[width=0.75\textwidth]{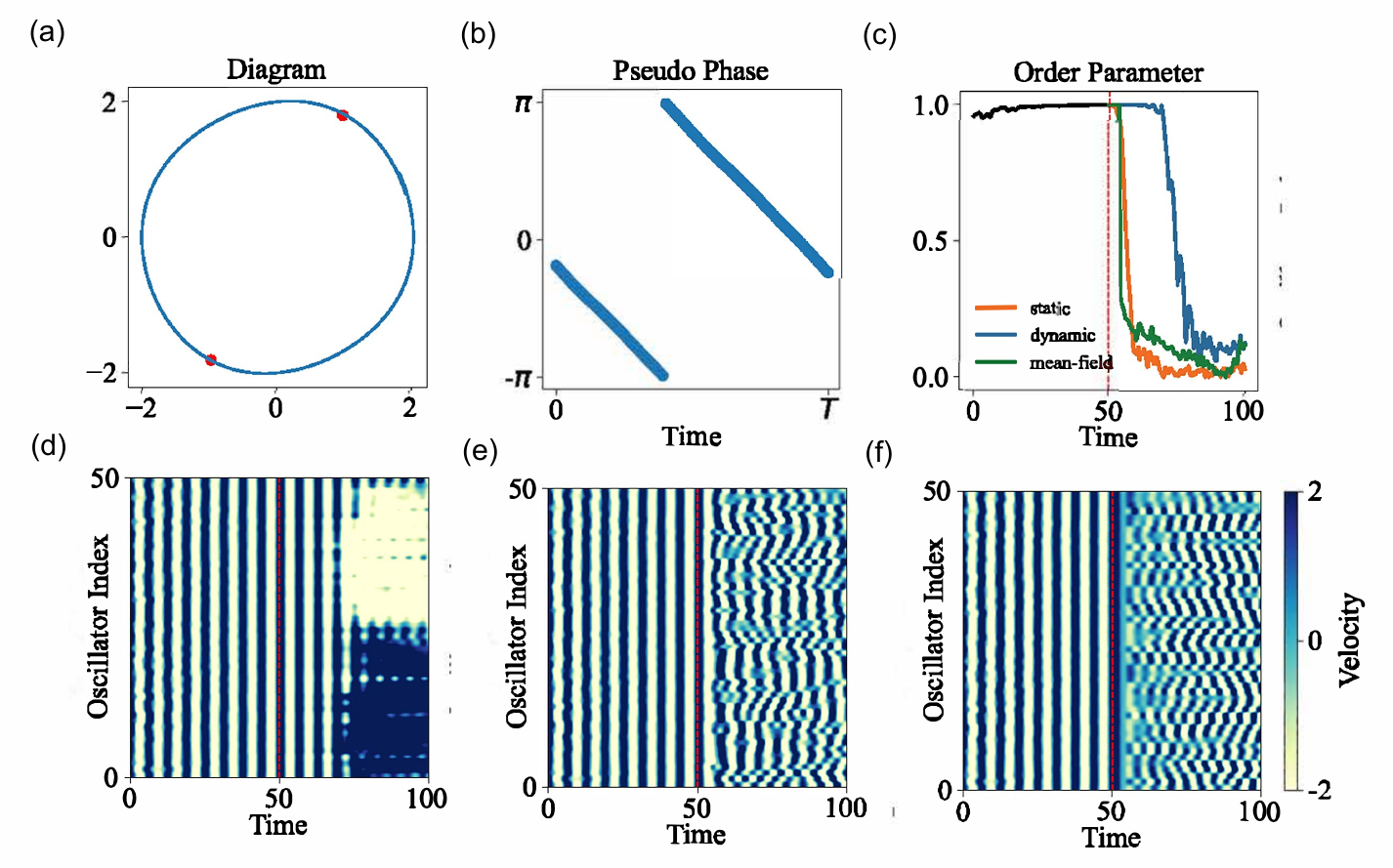}
	\vskip -0.17in
    \caption{Dynamics for system~\eqref{eq vdp main} where 50 homogeneous Van der Pol oscillators are coupled in
 a Watts-Strogatz network, before and after the control is switched
 on at $T = 50$. (a) Phase diagram of the original limit cycle in the ambient space, red points are $(A,B)$. (b) The pseudo phase uniformly varies along the limit cycle. (c) The order parameter before and after the control is applied. We consider three kinds of control: static law, dynamic law, mean-field adaptive control. (d)-(f)  Each colour represents the velocity of each oscillator before and after the control is applied.
     }
\vskip -0.15in
\end{figure}

To proceed, we consider a heterogeneous network in which the damping parameters $\mu_i$ are independently sampled from a uniform distribution $\mathcal{U}(0.03,0.17)$. The results shown in figure~\ref{fig_vdp2} demonstrate that the proposed desynchronisation algorithms remain effective in heterogeneous networks.

We further investigate the influence of network topology on the desynchronisation performance of different control laws. Figure~\ref{fig_vdp3} illustrates the topological structures of five representative networks in the top row. For each network, the desynchronisation effects of three control laws are examined under homogeneous (middle row) and heterogeneous (bottom row) damping parameters $\mu_i$, respectively. In all cases, the coupled oscillators are successfully desynchronised. In some instances, however, the order parameter does not converge exactly to zero, which can be attributed to approximation errors introduced by the pseudo-phase representation.

\begin{figure}\label{fig_vdp2}
	\centering
    \includegraphics[width=1.0\textwidth]{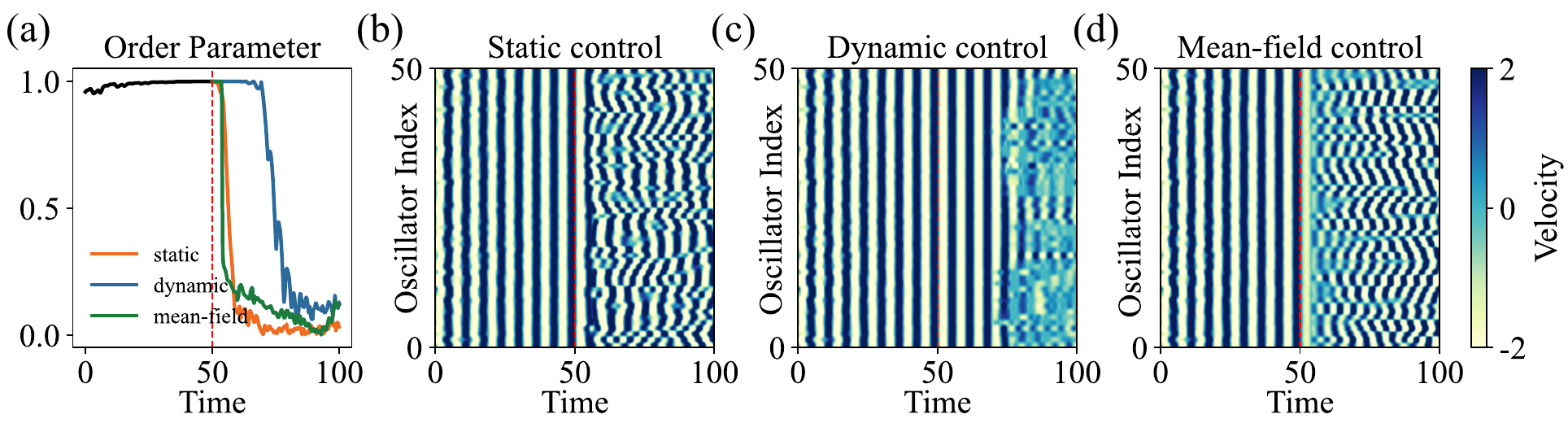}
	\vskip -0.17in
    \caption{Dynamics for system~\eqref{eq vdp main} where 50 heterogeneous Van der Pol oscillators are coupled in
 a Watts-Strogatz network, before and after the control is switched
 on at $T = 50$. (a) The order parameter before and after the control is applied. We consider three kinds of control: static law, dynamic law, mean-field adaptive control. (b)-(d)  Each colour represents the velocity of each oscillator before and after the control is applied. For the network's heterogeneity, we independently sample $50$ damping parameters $\mu_i,i=1,...,50$ from a uniform distribution $\mathcal{U}(0.03,0.17)$.
}
\vskip -0.05in
\end{figure}

\begin{figure}\label{fig_vdp3}
	\centering
    \includegraphics[width=1.0\textwidth]{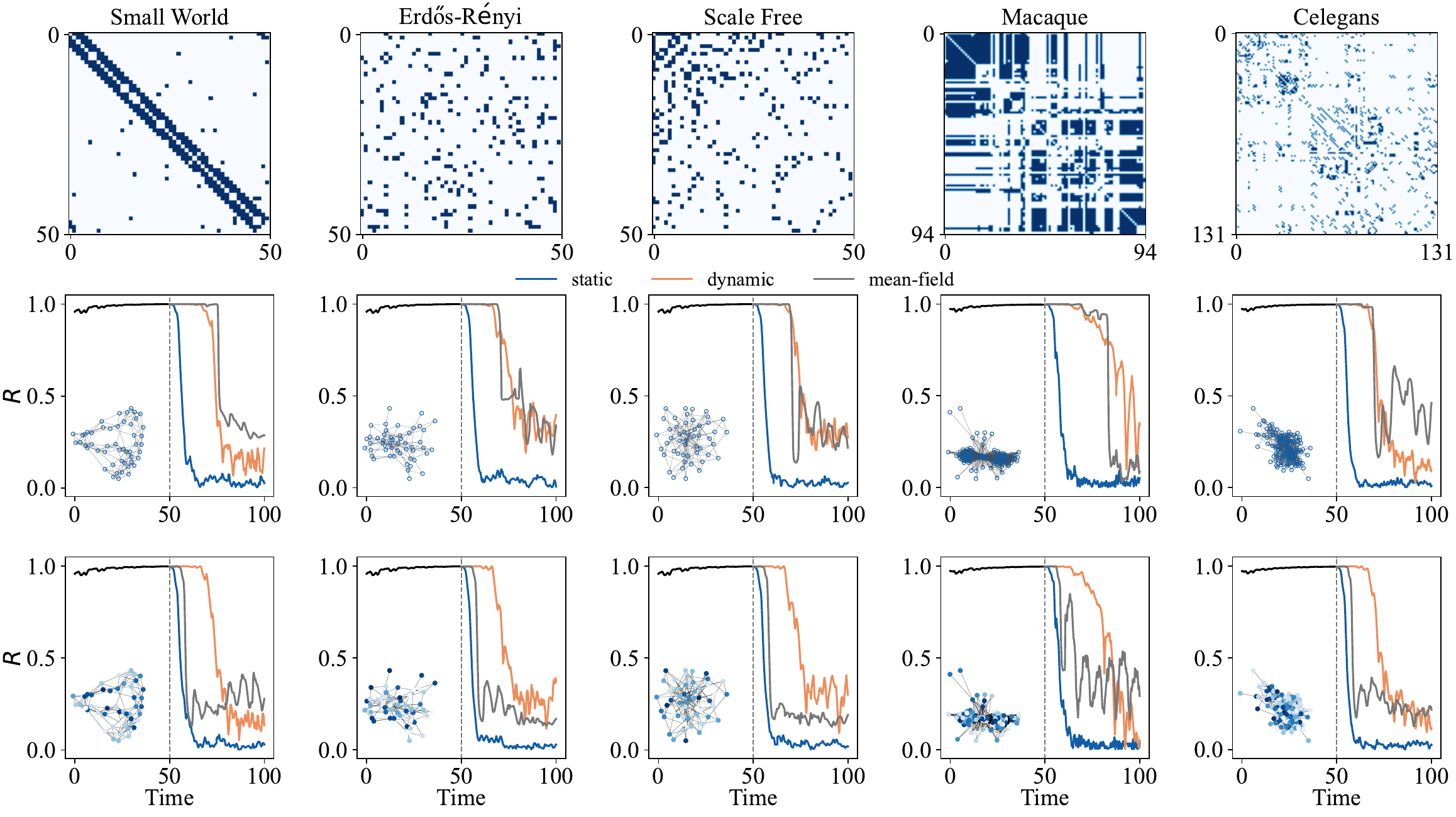}
	\vskip -0.1in
    \caption{(top row) Topological structures of five networks.  Order parameters of dynamics under different control laws for homogenous (middle row) and heterogenous (bottom row) oscillations, wherein the subplots show the network structure with colour representing the strength of the parameter $\mu_i$ in each node. The controllers are applied since time point $t=50$. For each network structure, we independently sample $N$ damping parameters $\mu_i,i=1,...,N$ from a uniform distribution $\mathcal{U}(0.03,0.17)$, where $N$ is the size of the network.
     }
\vskip -0.15in
\end{figure}

In the above numerical studies, we observe that the transient time required to drive the system from synchronisation to desynchronisation differs markedly among different control schemes. This discrepancy arises because the dynamic controllers do not exert an immediate effective influence: starting from zero initial conditions, they must first evolve into an effective region of the control space during the transient period before significantly impacting the collective dynamics. To systematically investigate how the control parameters $\rho$ and $\kappa$ in Eqs.~\eqref{eq static law heterogeneous},\eqref{eq dynamic law heterogeneous},\eqref{eq dynamic mean field final} affect the transient behaviour, we quantify the transient time as the interval between the activation of control and the first instance at which the order parameter satisfies $|R| = 0.1$, under different combinations of $\rho$ and $\kappa$. Figure~\ref{fig_vdp transient} shows that increasing the control gain $\kappa$ substantially accelerates the desynchronisation process, whereas excessively large values of $\rho$ may lead to practical failure of desynchronisation. This observation is consistent with the theoretical interpretation: the parameter $\rho$ governs the decay rate of the auxiliary control variable $u$, so overly rapid decay suppresses the effective control action, while a larger control strength $\kappa$ enhances the descent of the synchronisation energy and thus promotes faster desynchronisation.

\begin{figure}\label{fig_vdp transient}
	\centering
    \includegraphics[width=0.75\textwidth]{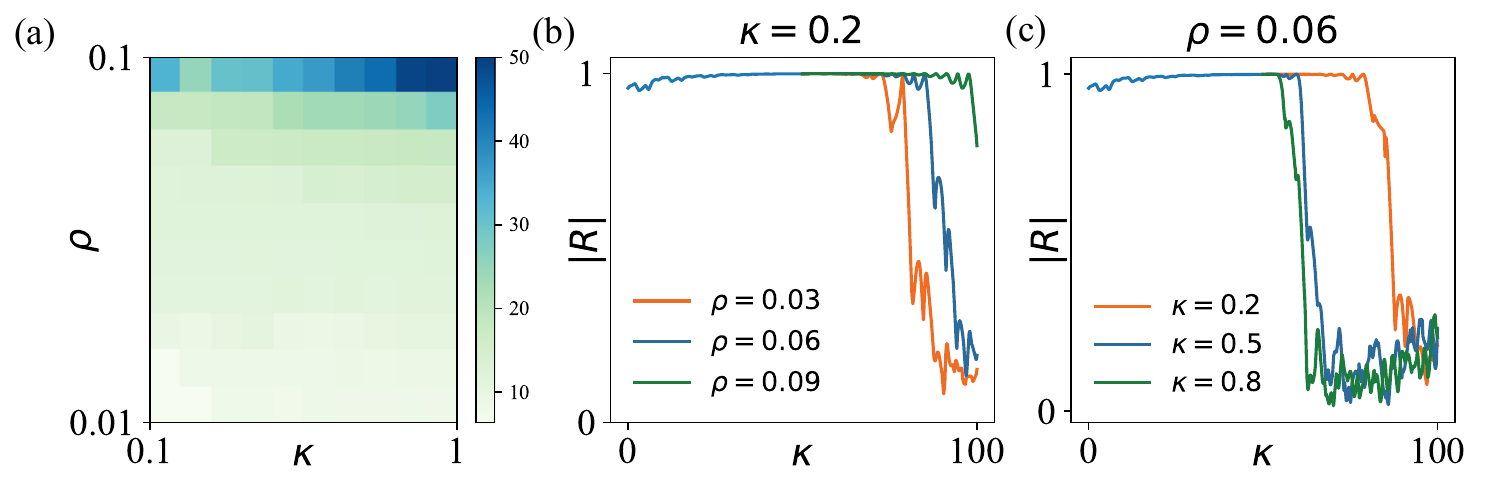}
    \caption{(a) Heatmap of the transient time under different parameters. (b) Larger $\rho$ increases the transient time from synchronisation to desynchronisation under fixed $\kappa$, (c) while larger $\kappa$ decreases the  transient time under fixed $\rho$.
     }
\vskip -0.15in
\end{figure}

To further enhance the practical applicability of the proposed control laws, we incorporate the \emph{event-triggered mechanism} introduced in Theorem~\ref{thm general system etc heter}. This mechanism substantially reduces the frequency of control updates while preserving the overall stability of the network.
For numerical validation, we consider a homogeneous oscillatory network with a Watts-Strogatz network structure, and examine both the triggering behaviour and the resulting desynchronisation performance of each control law.  The results shown in figure~\ref{fig_vdp etc} demonstrate that the event-triggered implementation achieves desynchronisation performance comparable to that of the continuous control laws (refer to figure~\ref{fig_vdp1}), while significantly reducing the intervention frequency and  communication costs.

\begin{figure}\label{fig_vdp etc}
	\centering
    
    \includegraphics[width=0.75\textwidth]{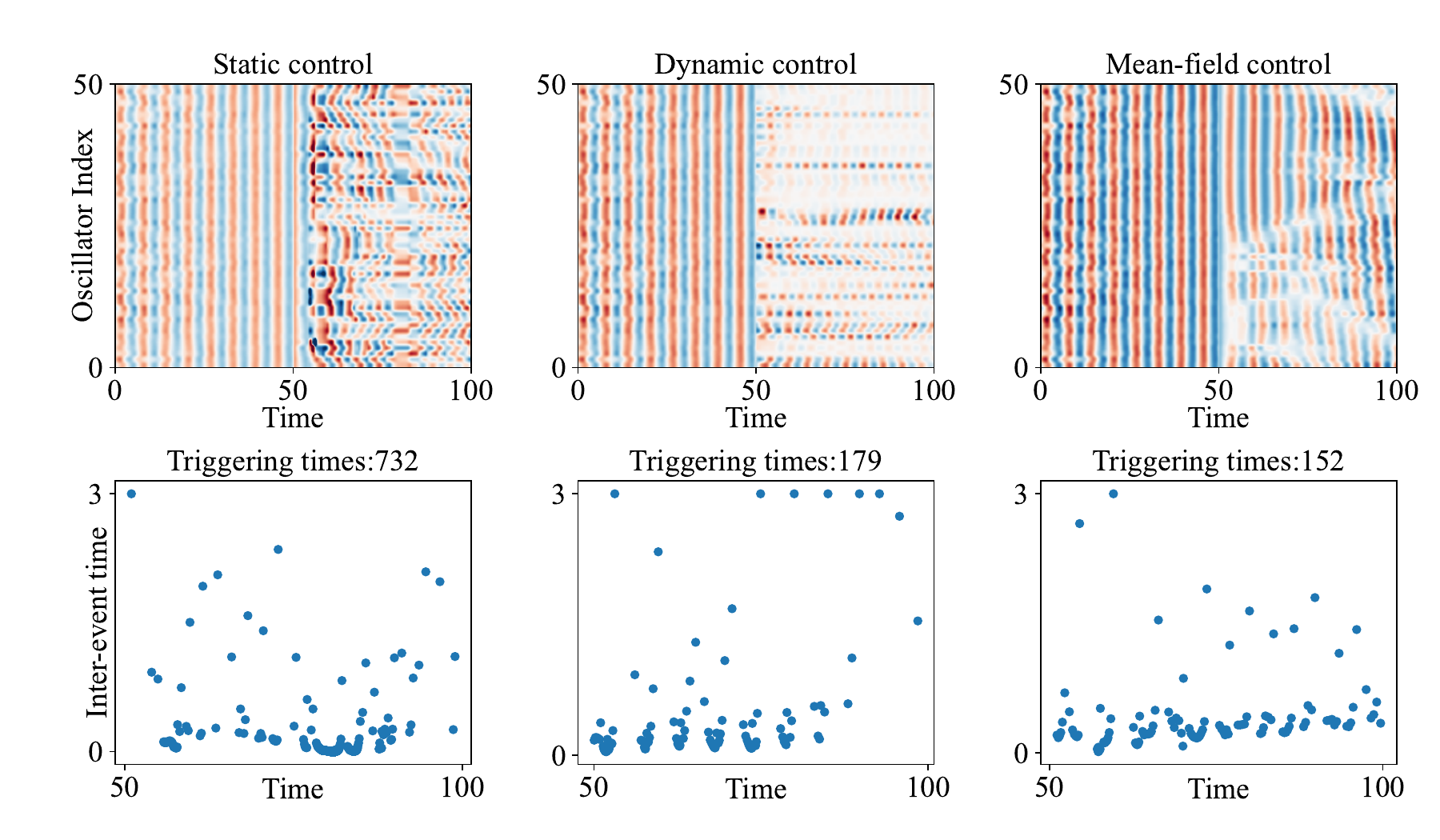}
	\vskip -0.17in
    \caption{Event-triggered desynchronisation of coupled Van der Pol oscillators using pseudo-phase feedback.
Dynamics of system (6.1) with $
N=50$ homogeneous oscillators on a Watts–Strogatz network. Control is activated at $
T=50$.
Top row: Spatiotemporal evolution of oscillator velocities under event-triggered static, dynamic, and mean-field control.
Bottom row: Inter-event intervals for each control law, demonstrating strictly positive dwell times and a substantial reduction in control updates.
     }
\vskip -0.15in
\end{figure}

    \subsection{Coupled three-dimensional oscillators} 
    In this section, we validate the effectiveness of the proposed control strategies to more complex scenarios. We consider the coupled three-dimensional Goodwin oscillators~\cite{goodwin1965oscillatory}, which is expressed as
 \begin{equation}\label{eq 3d main}
 \begin{aligned}
        \dot{x}_i&= \frac{\alpha}{1+z^n}-\beta x+K\sum_{j=1}^NA_{ij}x_j+u_i,\\
        \dot{y}_i &= \gamma x_i-\beta y_i,\\
        \dot{z}_i&=\gamma y_i-\beta z_i, i=1,...,N,
 \end{aligned}    
 \end{equation}
where $x_i$ denotes the concentration of messenger RNA (mRNA), whose production is repressed by the downstream protein $z_i$ through a Hill-type negative feedback, and $u_i$ is the external control input. The variable $y_i$ represents an intermediate protein translated from $x_i$, mediating the regulatory delay in the feedback loop. The variable $z_i$ corresponds to the final functional protein that inhibits the transcription of $x_i$ and closes the negative feedback loop, playing a central role in shaping the oscillatory and phase dynamics of the system. The nonlinear repression term with Hill coefficient $n$ introduces an effective delay in the negative feedback loop, giving rise to self-sustained oscillations through a Hopf bifurcation. Throughout the simulations, the parameters are chosen as
$
K = 0.4,~n = 9,  ~\alpha = 5.0, ~ \gamma = 2.0, ~ \beta= 1.2.$
Under these parameters, the isolated Goodwin system admits a stable and smooth limit cycle without sharp transitions or intrinsic fast--slow separation. For the network structure $(A_{ij})$, we consider the Watts-Strogatz network with $N=50$.  The uncontrolled coupled oscillators are synchronous over the limit cycle. 

\begin{figure}
	\centering
    \subfigure{\label{fig_3d1a}}
    \subfigure{\label{fig_3d1b}}
    \subfigure{\label{fig_3d1c}}
        \subfigure{\label{fig_3d1d}}
    \subfigure{\label{fig_3d1e}}
    \subfigure{\label{fig_3d1f}}
    \includegraphics[width=0.75\textwidth]{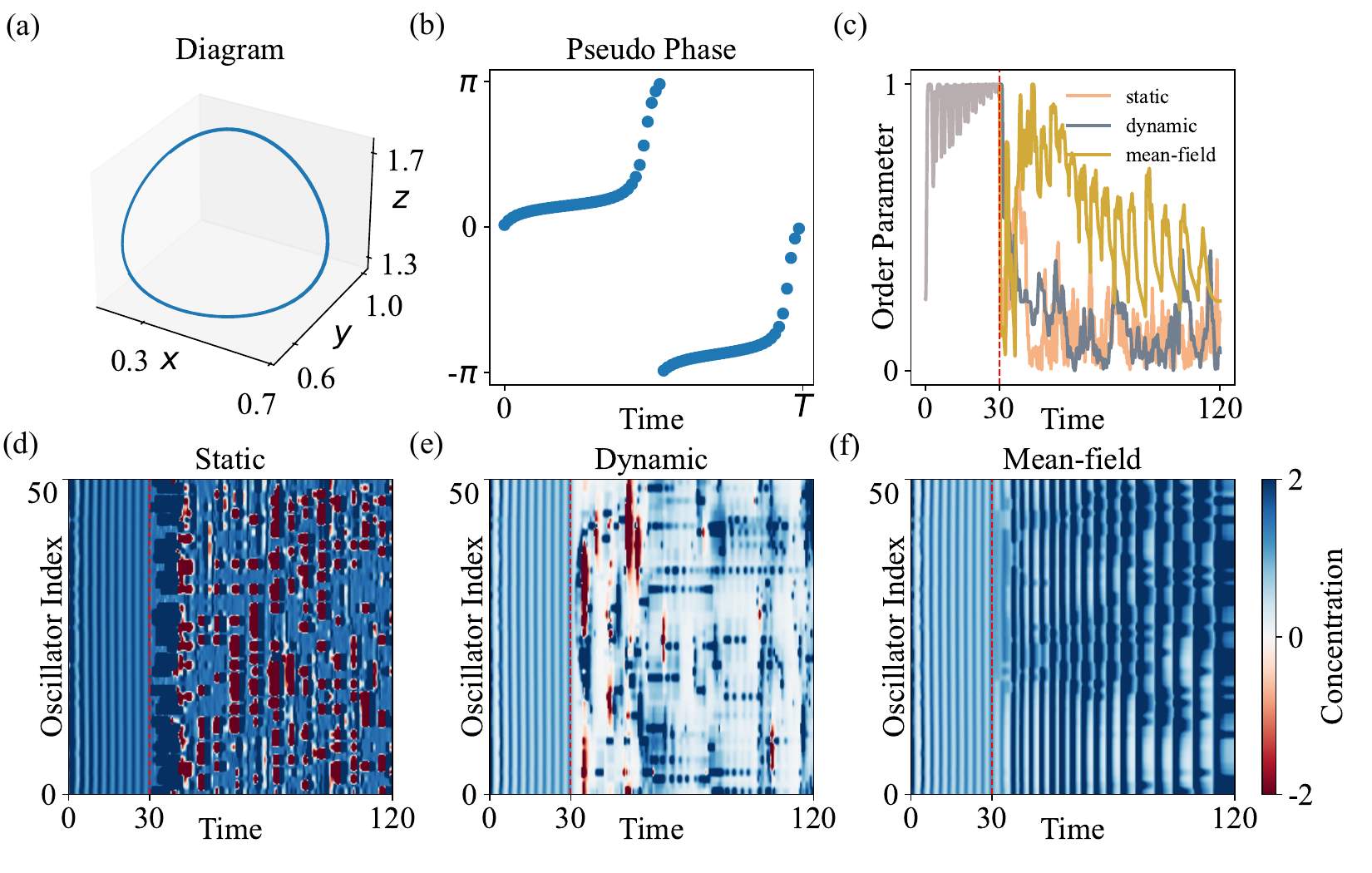}
	\vskip -0.17in
    \caption{Dynamics for system~\eqref{eq 3d main} where 50 homogeneous Goodwin oscillators are coupled in
 a Watts-Strogatz network, before and after the control is switched
 on at $T = 30$. (a) Phase diagram of the original limit cycle in the ambient space, red points are $(A,B)$. (b) The pseudo phase uniformly varies along the limit cycle. (c) The order parameter before and after the control is applied. We consider three kinds of control: static law, dynamic law, mean-field adaptive control. (d)-(f)  Each colour represents the concentration of each oscillator's mRNA before and after the control is applied.
     }
\vskip -0.15in
\label{fig_3d_homo}
\end{figure}
The geometric phase is constructed by projecting the three-dimensional trajectory onto a two-dimensional plane following the procedure introduced in Section~\ref{sec pseudo}, and defining the phase as the polar angle of the projected trajectory. Owing to the smooth geometry of the Goodwin limit cycle, the resulting pseudo phase evolves approximately uniformly along the orbit, in contrast to fast--slow spiking oscillators.

Figure~\ref{fig_3d_homo} illustrates the desynchronisation performance of the static, dynamic, and mean-field adaptive control laws applied to the coupled Goodwin oscillators. The control is activated at time $T=30$. Prior to control activation, the oscillators rapidly synchronise, as indicated by the order parameter approaching unity. After the control is switched on, all three control strategies effectively suppress synchronisation and drive the order parameter toward low values. Among them, the static and mean-field controllers induce an immediate decay of collective synchrony, while the dynamic controller exhibits a short transient period due to its internal adaptive dynamics before achieving comparable desynchronisation performance. The spatiotemporal plots further confirm that phase coherence is destroyed across the network while individual oscillations remain bounded.

To reduce communication cost, the control laws are further implemented under the event-triggered mechanism described in Section~\ref{sec etc}. Figure~\ref{fig_3d etc} presents the corresponding event-triggered desynchronisation results. The top panels show the time evolution of the global order parameter under static, dynamic, and mean-field control, while the bottom panels depict the inter-event intervals associated with each control strategy. In all cases, the event-triggered controllers achieve desynchronisation performance comparable to the continuous-time implementations, while substantially reducing the number of control updates. Moreover, the inter-event times remain finite and strictly positive throughout the simulations, thereby excluding Zeno behaviour.

Together, these results demonstrate that the Goodwin oscillator provides a representative three-dimensional limit-cycle system for which geometric phase-based control remains effective. The absence of pronounced fast--slow structure allows the projection-based phase to serve as a reliable surrogate for phase feedback, in contrast to spiking neuronal models. This example complements the two-dimensional experiments and highlights a regime in which geometric phase constructions are sufficient, thereby motivating the development of more general dynamical phase formulations for fast--slow oscillatory systems.

\begin{figure}
	\centering
    \subfigure{\label{fig_3d1a}}
    \subfigure{\label{fig_3d1b}}
    \subfigure{\label{fig_3d1c}}
        \subfigure{\label{fig_3d1d}}
    \subfigure{\label{fig_3d1e}}
    \subfigure{\label{fig_3d1f}}
    \includegraphics[width=0.75\textwidth]{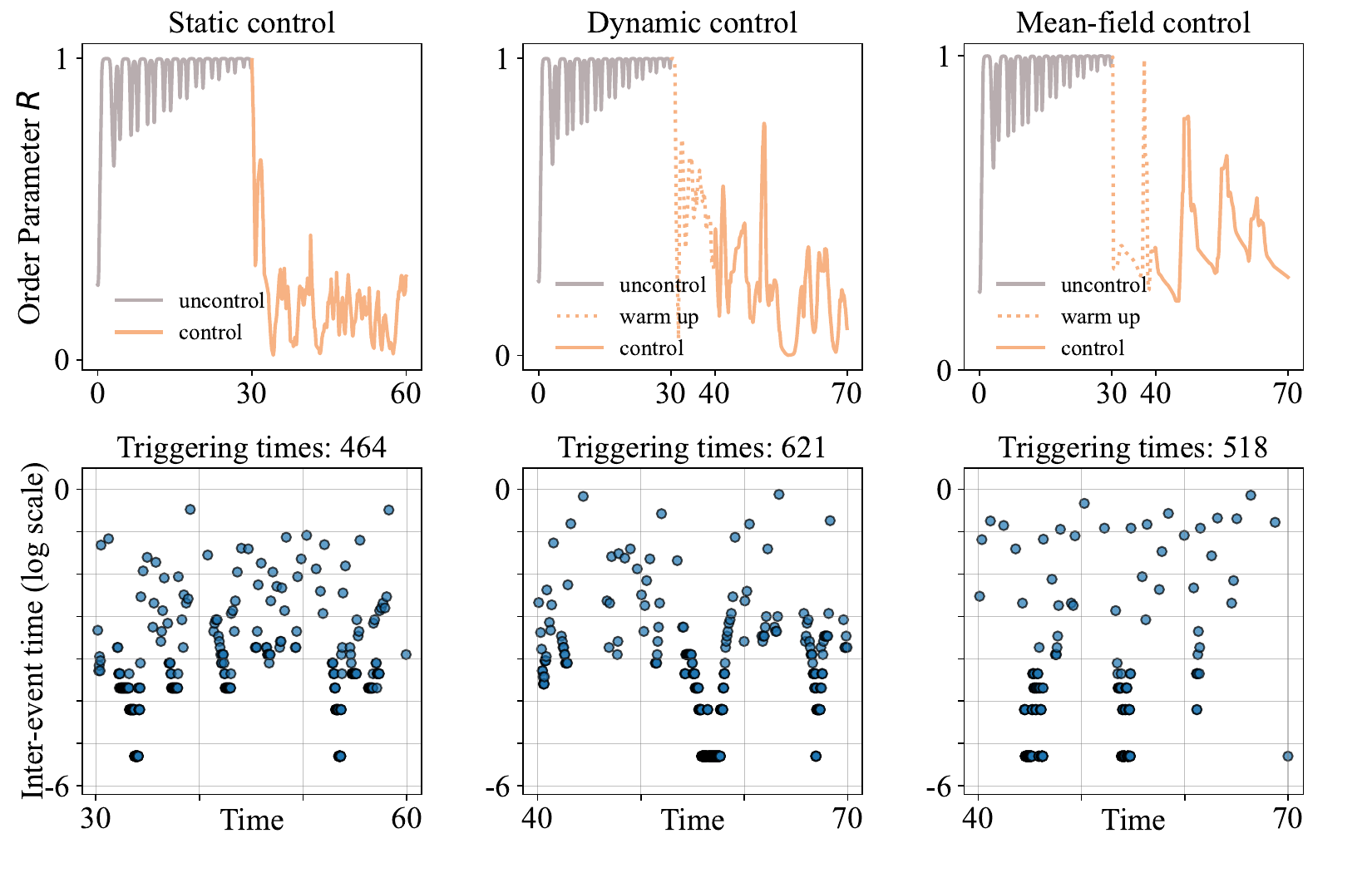}
	\vskip -0.17in
    \caption{Event-triggered desynchronisation of coupled three-dimensional Goodwin oscillators. Top row: time evolution of the global order parameter under event-triggered implementations of static, dynamic, and mean-field control laws, where control is activated at $T=30$.
Bottom row: inter-event intervals associated with each control strategy, showing strictly positive triggering times and thereby excluding Zeno behaviour.
}
\vskip -0.15in
\label{fig_3d etc}
\end{figure}

\subsection{Partial control on spatially structured brain networks}

We finally assess the performance of the proposed desynchronisation framework on realistic brain networks with explicit spatial embedding and heterogeneous connectivity. In contrast to the fully controlled synthetic networks studied above, we now consider \emph{partial control} scenarios, in which control inputs are applied only to a subset of nodes, while the remaining nodes evolve autonomously under network interactions. This setting is motivated by practical constraints in neuromodulation, such as deep brain stimulation (DBS), where stimulation can be delivered only to a limited number of anatomical regions.

Here we consider the controlled Van der Pol oscillators distributed on two experimentally derived brain networks mentioned in Section~\ref{sub sec vdp}: a macaque cortical connectivity network within a single hemisphere, consisting of 94 cortical regions connected by weighted long-range projections, and a local neuronal network of \emph{C.~elegans}, comprising 131 frontal neurons with anatomically identified synaptic connectivity. Both networks exhibit strong heterogeneity in their degree distributions and nontrivial spatial organisation, making them representative test cases for large-scale biological networks. The mathematical setting are the same as that in  Section~\ref{sub sec vdp}.

To model spatially constrained actuation, control-node selection is formulated as a two-stage process that decouples anatomical localization from topological influence. Nodes are first grouped into spatially continuous clusters based solely on their Euclidean coordinates, defining anatomically admissible regions for control. Topological properties are then evaluated within each cluster, and control is selectively applied to structurally influential nodes within a hub-enriched region. This procedure yields compact control sets that are spatially localized yet exert disproportionate influence on network dynamics. 

We consider both the homogeneous and heterogeneous oscillations under three proposed control laws. The resulting desynchronisation dynamics under partial control are shown in figure~\ref{fig_partial} for both networks. Despite the fact that only a fraction of nodes are directly actuated, the global order parameter decreases rapidly following control activation, indicating effective suppression of collective synchrony. Importantly, desynchronisation is not confined to the controlled subset: spatiotemporal visualisations reveal that phase coherence is progressively disrupted throughout the network, including the uncontrolled nodes, via coupling-mediated propagation. These results demonstrate that targeting structurally influential nodes is sufficient to induce network-wide desynchronisation in spatially heterogeneous brain networks.

To further reduce actuation and communication costs, we implement the partial control strategy under the event-triggered mechanism to the homogeneous \emph{C.~elegans} neuronal network. Figure~\ref{fig_partial_etc} shows the corresponding order parameter evolution, and the distribution of the triggering times under each kind of controller. As in the fully controlled cases, the event-triggered implementation achieves desynchronisation performance comparable to continuous-time control, while substantially reducing the number of control updates. The inter-event times remain strictly positive throughout the simulations, confirming the exclusion of Zeno behaviour. Importantly, effective desynchronisation is maintained even though control actions are sparse in time and limited to a subset of nodes.

These results highlight the effectiveness of geometric phase-based partial control on realistic brain networks and provide a conceptual link to DBS applications. In clinical settings, stimulation is typically delivered to a small number of nuclei, yet therapeutic effects emerge at the level of distributed brain networks~\cite{wang2023desynchronizing}. The present results demonstrate that, by exploiting network structure and phase interactions, desynchronisation of large-scale oscillatory activity can be achieved without direct control of all nodes.

\begin{figure}
	\centering
    
    \includegraphics[width=0.75\textwidth]{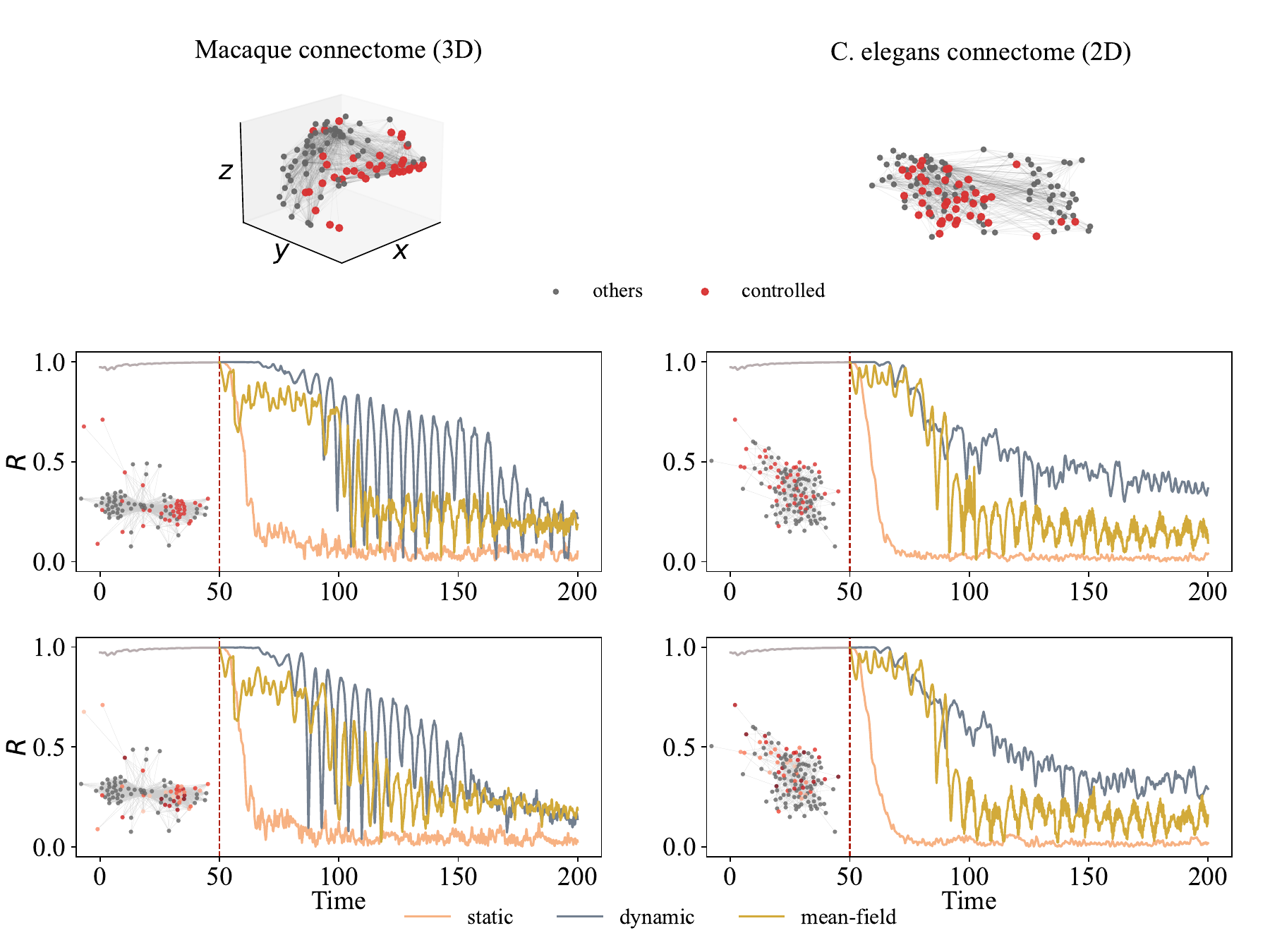}
	\vskip -0.17in
    \caption{Desynchronisation of coupled Van der Pol oscillators under partial control. Top row: Illustration of controlled nodes in realistic brain networks. Desynchronisation results of dynamics under different control laws for homogenous (middle row) and heterogenous (bottom row) oscillations, wherein the subplots show the network structure with red colour representing the strength of the parameter $\mu_i$ in each controlled node. The controllers are applied since time point $t=50$.
     }
\vskip -0.15in
\label{fig_partial}
\end{figure}

\begin{figure}
	\centering
    
    \includegraphics[width=0.75\textwidth]{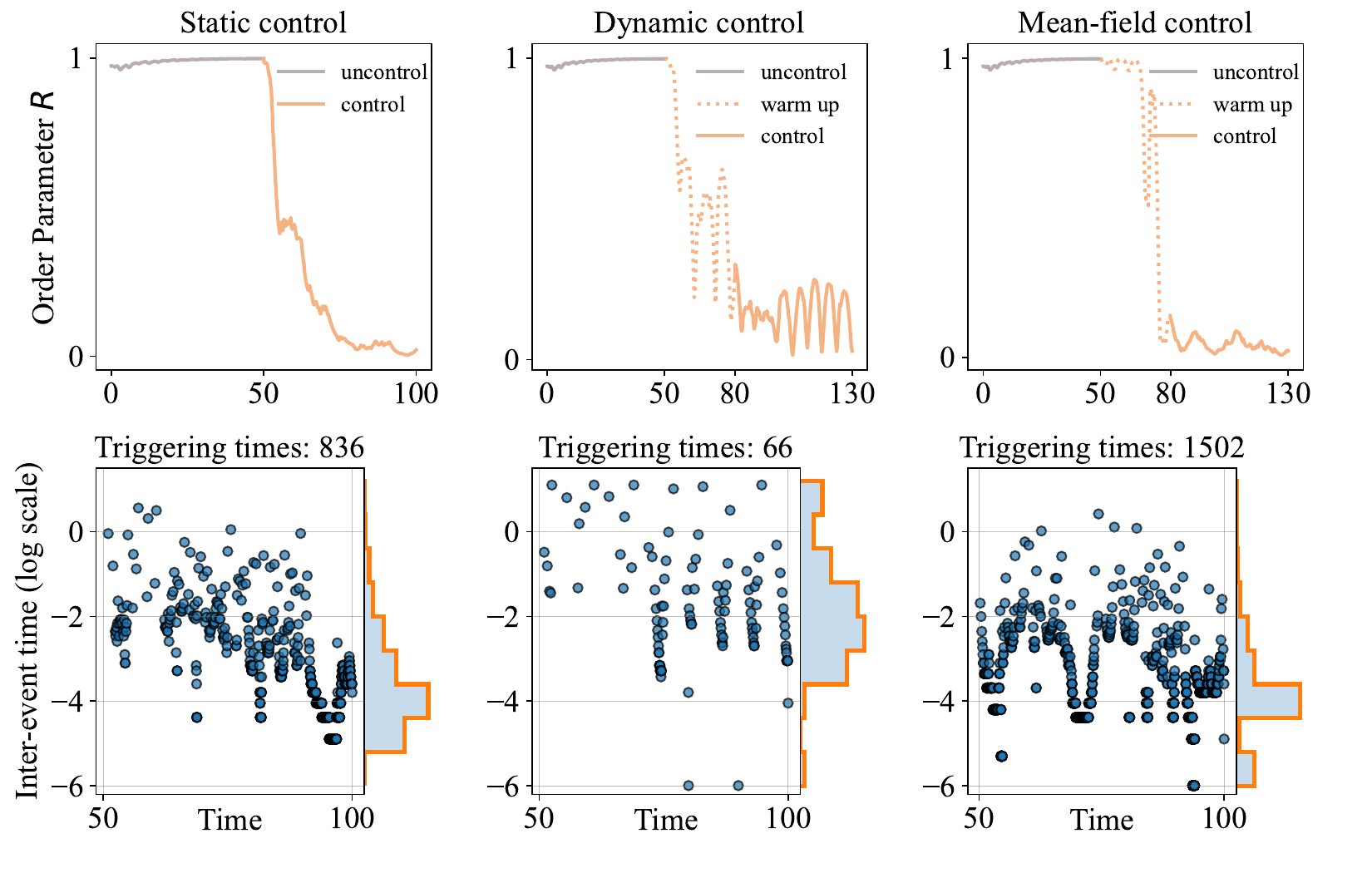}
	\vskip -0.17in
    \caption{Event-triggered desynchronisation of coupled Van der Pol oscillators distributed over a homogeneous \emph{C.~elegans} neuronal network. Top row: time evolution of the global order parameter under event-triggered implementations of static, dynamic, and mean-field control laws, where control is activated at $T=50$.
Bottom row: In each panel the left subfigure shows the inter-event intervals associated with each control strategy, and the right subfigure summarises the distribution of the inter-event intervals.
}
\vskip -0.15in
\label{fig_partial_etc}
\end{figure}

  \section{Conclusion}\label{sec conclusion}
 This paper developed a general event-triggered control framework for stabilising desynchronisation in coupled limit-cycle oscillators. By formulating desynchronisation as the zero point of synchronisation energy, we constructed static control with rigorous stability guarantees, as well as dynamic and mean-field control laws for flexibility. We further incorporated an event-triggered mechanism to reduce communication cost while excluding Zeno behaviour via explicit dwell-time bounds. To address the intractability of exact phase reduction, we introduced a projection-based pseudo-phase construction that enables order-parameter feedback directly from state measurements.

Numerical experiments on synthetic networks and spatially structured brain networks demonstrate the robustness of the proposed approach even using partial control. These results provide a theoretically grounded and practically implementable route toward sparse desynchronisation control in large-scale oscillatory systems in DBS scenarios. Future work will extend the framework to fast–slow dynamics by integrating machine learning based phase reduction methods~\cite{zhang2026ptolemy} and control methods~\cite{zhangfessnc}, and extend the underlying dynamcis to systems with stochastic perturbations and time delays~\cite{zhang2023sync}.

\clearpage



\bibliographystyle{siamplain}
\bibliography{references}
\end{document}